\documentclass[11pt]{amsart}
\usepackage{amssymb, amsmath, amsfonts, amsthm, graphics}
\usepackage{hyperref}
\usepackage{mathtools}
\usepackage{tikz-cd}
\usepackage{stmaryrd}
\usepackage{mathrsfs}

\usepackage[hmargin=1.4 in, vmargin = 1.4 in, lmargin=1.3 in, rmargin=1.3 in]{geometry}
\usepackage{enumerate}
\newtheorem{thm}{Theorem}[section]
\newtheorem{cor}[thm]{Corollary}
\newtheorem{lem}[thm]{Lemma}
\newtheorem{defn}[thm]{Definition}
\newtheorem{prop}[thm]{Proposition}
\newtheorem{ex}[thm]{Example}
\newtheorem{remark}[thm]{Remark}
\DeclareMathOperator{\Hom}{Hom}

\DeclareMathOperator{\Spec}{Spec}

\DeclareMathOperator{\Proj}{Proj}

\DeclareMathOperator\codim{codim}

\DeclareMathOperator{\an}{an}

\DeclareMathOperator{\vol}{vol}
\DeclareMathOperator{\na}{NA}
\DeclareMathOperator{\trop}{trop}
\DeclareMathOperator{\inte}{Int}

\renewcommand{\AA}{\mathbb{A}}

\newcommand{\CC}{\mathbb{C}}

\newcommand{\GG}{\mathbb{G}}

\newcommand{\KK}{\mathbb{K}}
\newcommand{\LL}{\mathbb{L}}

\newcommand{\NN}{\mathbb{N}}

\newcommand{\QQ}{\mathbb{Q}}
\newcommand{\RR}{\mathbb{R}}

\newcommand{\TT}{\mathbb{T}}

\newcommand{\ZZ}{\mathbb{Z}}

\newcommand{\cC}{\mathcal{C}}
\newcommand{\cD}{\mathcal{D}}

\newcommand{\cF}{\mathcal{F}}

\newcommand{\cH}{\mathcal{H}}
\newcommand{\cI}{\mathcal{I}}

\newcommand{\cL}{\mathcal{L}}

\newcommand{\cO}{\mathcal{O}}

\newcommand{\cR}{\mathcal{R}}

\newcommand{\cV}{\mathcal{V}}

\newcommand{\cX}{\mathcal{X}}

\newcommand{\fa}{\mathfrak{a}}

\newcommand{\fm}{\mathfrak{m}}

\newcommand{\la}{\langle}
\newcommand{\ra}{\rangle}
\newcommand{\lb}{\llbracket}
\newcommand{\rb}{\rrbracket}

\newcommand{\Addresses}{{% additional braces for segregating \footnotesize
  \bigskip
  \footnotesize

  \textsc{Dept of Mathematics, University of Michigan,
    Ann Arbor, Michigan 48109-1043, USA}\par\nopagebreak
  \textit{E-mail address}: \texttt{yueqiaow@umich.edu}

}}

\begin{document}

\title{Volume and Monge-Amp\`ere energy on polarized affine varieties}
	
\author{Yueqiao Wu}

\begin{abstract}
Let $(X, \xi)$ be a polarized affine variety, i.e. an affine variety $X$ with a (possibly irrational) Reeb vector field $\xi$. We define the volume of a filtration of the coordinate ring of $X$ in terms of the asymptotics of the average of jumping numbers. When the filtration is finitely generated, it induces a Fubini-Study function $\varphi$ on the Berkovich analytification of $X$. In this case, we define the Monge-Amp\`ere energy for $\varphi$ using the theory of forms and currents on Berkovich spaces developed by Chambert-Loir and Ducros, and show that it agrees with the volume of the filtration. In the special case when the filtration comes from a test configuration, we recover the functional defined by Collins-Sz\'ekelyhidi and Li-Xu.

\end{abstract}

\maketitle
	
\tableofcontents

\section{Introduction}
We work over an algebraically closed field $\KK$ of characteristic zero. Let $V$ be a Fano variety. The Yau-Tian-Donaldson conjecture, proved in \cite{CDS, tian}, and more recently using a variational approach in \cite{BBJ, Li19, LXZ}, states that $V$ admits a weak K\"ahler-Einstein metric if and only if it's K-polystable. This was generalized in \cite{CS19} to K\"ahler cones, relating the existence of a Ricci flat K\"ahler cone metric to K-stability. As part of the variational approach, a Monge-Amp\`ere energy was introduced for a test configuration, and later an algebraic formulation was given in \cite{LX18, LWX} computing the limit slope of the Monge-Amp\`ere energy.

Let $(X = \Spec R, \TT, \xi)$ be a polarized affine variety, i.e. $X$ is a normal affine variety, $\TT$ is a torus of automorphisms with a unique fixed point $0\in X$ with ideal $\fm$, and $\xi$ is a vector field generating a subtorus of $\TT$ and acting with positive weights. We will usually omit $\TT$ and write $(X, \xi)$ to mean the polarized affine variety. Let $R = \bigoplus_{\substack{\alpha \in \Lambda}} R_\alpha$ be the weight decomposition of $R$ with respect to $\TT$, where $\Lambda = \{\alpha\in \Hom(\TT, \GG_m): R_\alpha\neq 0\}$. Let $n$ be the dimension of $X$.

If $(\cX, \xi, \eta)$ is a $\QQ$-Gorenstein test configuration for $(X, \xi)$, then the following functional
\[\frac{D_{-\eta} \vol_{\cX_0}(\xi)}{\vol(\xi)}\]
is defined for $(\cX, \xi, \eta)$ in \cite{CSIrregular, LX18, LWX} as the limit slope of the Monge-Amp\`ere energy defined in \cite{CS19}.
Here $$\vol(\xi) = \lim\limits_{m\to \infty}\frac{\sum_{\substack{\la \alpha, \xi\ra\le m}}\dim R_\alpha}{m^n/n!}, $$
and
$$D_{-\eta} \vol_{\cX_0}(\xi) = \left.\frac{d}{d\varepsilon} \right|_{\varepsilon = 0} \vol_{\cX_0}(\xi - \varepsilon \eta)$$ 
is the directional derivative of the volume function along the direction $-\eta$.

In this paper, we aim to take two different perspectives towards this functional, and generalize it to a functional associated to a finitely generated filtration. 

First, following \cite{Nystrom, BHJ17, BlumJonsson}, we realize the energy functional as asymptotics of the average of jumping numbers of some filtration. More precisely, each test configuration $(\cX = \Spec \cR, \xi, \eta)$ gives rise to a filtration $\cF$ on $R$ as follows:
\[ \cF^\lambda R_\alpha = \{f\in R_\alpha : t^{-\lceil\lambda\rceil}\bar{f}\in \cR_\alpha\},\]
where $\bar{f}$ is the pullback of $f$ along the morphism $\cX \times_{\AA^1}\GG_m \cong X\times \GG_m\to X$. 
We write $$R_m = \bigoplus_{\substack{\la \alpha, \xi\ra \le m}} R_\alpha,$$
and 
$$\cF^\lambda R_m =  \bigoplus_{\la \alpha, \xi\ra \le m} \cF^{\lambda} R_\alpha, $$
for each integer $m\ge 1$.
Such filtrations are finitely generated filtration on $R$. The following constructions are thus made more generally for any finitely generated filtration $\cF$.
On each $R_m$, define the \emph{sequence of jumping numbers} 
$$0\leq a_{m,1} \leq \cdots \leq a_{m, N_m},$$
where $N_m = \dim_\KK R_m$, via
$$ a_{m,j} = a_{m,j}(\cF) = \inf\{\lambda \in \RR_+: \codim \cF^\lambda R_m \geq j\}, \  1\leq j\leq N_m.$$
Consider the average 
\[S_m(\cF) = \frac{1}{mN_m} \sum_{j=1}^{N_m} a_{m,j}.\]
Using an analog of the Okounkov bodies, we show that the limit
$S(\cF) = \lim\limits_{m\to \infty} S_m(\cF)$ exists. We define the \emph{volume} of the filtration $\cF$ to be $S(\cF)$.
In the case when $\xi$ is rational, and $(X, \xi)$ viewed as the cone over a polarized pair, we show that the volume of the filtration agrees with the similarly defined S-invariant defined in \cite{BlumJonsson} up to a constant, and it is shown in \cite{BHJ17} that this $S$-invariant serves as the non-Archimedean Monge-Amp\`ere energy. 
When the filtration is induced by a valuation, the volume also generalizes the S-invariant appeared in \cite{XuZhuang20}.

Next, we try to write the Monge-Amp\`ere energy as an integral of functions on the Berkovich analytification $X^{\an}$ of $X$ with respect to the trivial valuation on $\KK$. Similarly as in \cite{BHJ17}, each finitely generated filtration gives a
Fubini-Study function $\varphi$ of the form
\[\varphi = \max_{1\le j\le N} \{\frac{\log|f_j|+\chi(f_j)}{\la \alpha_j, \xi\ra}, f_j\in R_{\alpha_j}\},\]
where $\chi: R \to \RR$ is a norm of finite type induced by the filtration, and $f_j, 1\leq j\leq N$ generate a $\fm$-primary ideal.
In particular, the trivial test configuration gives a function $\varphi_\xi$, which can be seen as an analog of the trivial metric in the global case (see \cite{BJ18}), or as an analog of the psh potential for the reference metric in the complex case. Note that for $\varphi, \varphi_\xi$ to be written as a finite max, we need that the filtration be finitely generated.
Following the formalism of forms and integrals on Berkovich spaces in \cite{CLD12}, we can define the Monge-Amp\`ere operator. For Fubini-Study functions $\varphi_1, \cdots, \varphi_{n-1}$, we define an atomic measure with finite support on the subset $\{\varphi_\xi = 0\}\subset X^{\an}$ to be 
	\[d'd''\varphi_1\wedge \cdots \wedge d'd''\varphi_{n-1}\wedge d'd''(\max\{\varphi_\xi, 0\}).\]
	The theory of Chambert-Loir and Ducros allows us to define the non-Archimedean Monge-Amp\`ere energy as
\[E^{\na}(\varphi) = \frac{1}{n\vol(\xi)}\sum_{j=0}^{n-1} \int_{X^{\an}\setminus\{0\}} (\varphi - \varphi_\xi) (d'd''\varphi)^j\wedge (d'd''\varphi_\xi)^{n-j-1}\wedge d'd''(\max\{\varphi_\xi, 0\}).\]
Note that this formula is identical to the one for complex Monge-Amp\`ere energy, see Remark \ref{FormulaEqual} for more details.

Our main result is

\begin{thm}\label{main} For a Fubini-Study function $\varphi$ induced by a finitely generated filtration $\cF$, we have
	\[E^{\na}(\varphi) = S(\cF),\]
	and when the filtration comes from a test configuration $(\cX,\xi, \eta)$ with $\eta$ in the Reeb cone of the torus $\TT\times \GG_m$, 
	\[E^{\na}(\varphi) =C(n) \frac{D_{-\eta} \vol_{\cX_0}(\xi)}{\vol(\xi)}\]
	up to a dimensional constant $C(n)$.
\end{thm}

The main strategy is to first prove the case when $\xi$ is rational. In this case, write $X = C(V, L)$ as the cone of a polarized pair $(V, L)$. The test configuration for $(X, \xi)$ corresponds to the cone of an ample test configuration $(\cV, \cL)$ for $(V, L)$. This gives us a Fubini-Study metric $\phi$ on $L^{\an}$. We show that $E^{\na} (\varphi) = E^{\na}(\phi)$. The general result then follows from approximating $\xi$ by rational ones.
We anticipate that this energy be extended to a bigger class of functions, e.g. the finite energy class, and other related non-Archimedean functionals be defined in our future work.

Finally, we look at the toric case. When $X$ is an affine toric variety of dimension $n$ associated to a cone $\sigma \subset N_\RR$, $\xi$ can be thought of as a vector in the interior of $\sigma$, and Berman defined in \cite{BermanToric} the Monge-Amp\`ere energy of a bounded convex function $\psi$ on $\sigma^\vee$ to be 
\[E(\psi) = \frac{1}{\vol(\xi)} \int_{P_\xi} \psi d\lambda,\]
where $P_\xi = \{u\in \sigma^\vee: \la u, \xi\ra =1\}$, and $d\lambda$ denotes the Lebesgue measure on $P_\xi$. Each filtration $\cF$ gives a function $\psi: \sigma^\vee \cap M \to \RR$ by $\psi(u) = \sup \{\lambda: \cF^\lambda R_u \neq 0\}$, and this further extends to a convex sublinear function $\tilde{\psi}$ on $\sigma^\vee$.

Our next result relates this energy to the $S$ invariant: 

\begin{thm}\label{1.2}
	With the notation as above,
	\[S(\cF) = C(n)E(\tilde{\psi})\]
	for some constant $C(n)$.
\end{thm}

\subsection*{Organization} This paper is organized as follows. In section 2, we briefly recall notions of valuations, test configurations, Berkovich analytifications, and the theory of forms and currents on Berkovich spaces. In section 3, we construct a local analog of Okounkov bodies associated to a filtration, using which we define the volume of a filtration. A first formulation of the NA Monge-Amp\`ere energy is given in section 4, and we show that it computes the limit slope of the Monge-Amp\`ere energy defined in \cite{CS19} if the filtration comes from a test configuration. Section 5 is devoted to another formulation of the energy using forms and currents on Berkovich spaces, and Theorem \ref{main} is proved here. Finally, section 6 studies similar formulations in the toric case.

\subsection*{Acknowledgement.} I would like to thank my advisor, Mattias Jonsson, for suggesting the problem and kindly sharing his ideas. I'm grateful to Chi Li for answering my questions on his papers. I thank Harold Blum, S\'ebastien Boucksom, Gabor Sz\'ekelyhidi, Chenyang Xu and Ziquan Zhuang for helpful comments on a preliminary version of the paper. I'm also grateful to the anonymous referee for very helpful comments.

\section{Background}

\subsection{Polarized affine varieties and test configurations}\label{quasiregular}
We work over an algebraically closed field $\KK$ of characteristic 0. Let $X = \Spec R$ be an $n$-dimensional normal affine variety(not necessarily irreducible).

Let $\TT$ be a split torus acting on $X$.
Set $N: =\Hom(\GG_m, \TT)$ to be the co-weight lattice, and let $M$ be the dual lattice. We then have a weight decomposition on $R$:
\[R = \bigoplus_{\alpha\in \Lambda} R_\alpha,  \]
where $\Lambda := \{\alpha \in M: R_\alpha\neq 0\}.$
A $\TT$ action on $X$ is \emph{good} if has a unique fixed point $0\in X$ that is in the closure of all $\TT$-orbits and $R_0= \KK$.  In this paper, we will always assume that the action is good.

We now set up some definitions introduced in \cite{CS19}.
\begin{defn}
	The Reeb cone of $X$ with respect to the $\TT$ action is 
	\[\cC := \{\xi\in N_\RR:=N\otimes \RR: \la \xi, \alpha\ra >0, \  \forall \alpha\in \Lambda\setminus\{0\}\}\]
	A vector $\xi\in \cC$ is called a Reeb vector field.
\end{defn}
\begin{defn}
	A polarized affine variety is a triple $(X, \TT, \xi)$ such that $\TT$ is a good action on $X$, and $\xi$ is a Reeb vector field. Further, we call a polarized affine variety $(X, \TT, \xi)$ quasi-regular if $\xi$ is rational, i.e., $\xi\in \NN_\QQ$. Otherwise, it's said to be irregular.
\end{defn}

\begin{remark}
This is more commonly known as a log Fano cone singularity when $X$ has klt singularities. We adopt a different name here since we do not assume $X$ has klt singularities.
\end{remark}

\begin{defn}
	For a Reeb vector field $\xi \in \cC$, its volume is defined to be
	\[\vol(\xi):=\lim\limits_{m\to \infty}\frac{\sum_{\la \alpha, \xi\ra\le m}\dim R_\alpha}{m^n/n!}.\]
\end{defn}

The above limit exists by \cite{CSIrregular}. 

\begin{defn}
	A (real)  semivaluation $v$ on $X$ is a map $v: R\to \RR \cup \{\infty\}$ such that 
	\begin{enumerate}
		\item $v(f+g)\ge \min\{v(f), v(g)\} ,  \forall f, g\in R$;
		\item $v(fg) = v(f)+v(g), \forall f, g\in R$;
		\item $v(0)=\infty;$
		\item $v|_{\KK^*} = 0$.
	\end{enumerate}
	A valuation $v$ is a semivaluation such that $v(f)=\infty$ if and only if $f=0.$
\end{defn}

When $X$ is reduced and irreducible,  a Reeb vector field $\xi$ induces a valuation on $X$, which we denote by $v_\xi$:
\[v_\xi(f) = \min_\alpha\{\la \xi, \alpha\ra: f=\sum_\alpha f_\alpha, f_\alpha\neq 0\}.\]   
In this case, it was noted in \cite{LWX} that the above volume agrees with the volume for the valuation $v_\xi$, as defined in \cite{ELS}.

We remark that $\xi$ plays the role of a polarization (see \cite{CSIrregular}). Indeed, when $\xi$ is rational, say $\xi\in \frac 1l N$, it generates a $\GG_m$ action. So one can think of $X\setminus\{0\}$ as the complement of the zero section in the total space of an ample orbiline bundle over some orbifold. In algebro-geometric terms, $X\setminus \{0\}$ is a Seifert $\GG_m$-bundle over a projective variety $V = \Proj R$. It's shown in \cite{kollar} that the coordinate ring $R$ with the new grading induced by the $\xi$ action can be viewed as the section ring of some $\QQ$-line bundle $L$ on $V$:
\[R = \bigoplus_\alpha R_\alpha = \bigoplus_k\left(\bigoplus_{\la \xi, \alpha\ra = \frac kl}R_\alpha\right) = \bigoplus_{k\in 
\ZZ} H^0(V, kL).\]
We refer to \cite{RossThomas, kollar} for more details.

\begin{defn}
	Let $(X, \TT, \xi)$ be a polarized affine variety. A test configuration is a quadruple $(\cX, \TT, \xi, \eta)$ with a flat projection $\pi: \cX\to \AA^1$ such that
	\begin{enumerate}
		\item $\cX = \Spec \cR$ is an affine variety
		\item $\eta$ is a $\GG_m$ action on $\cX$ lifting the usual action on $\AA^1$. 
		\item $\TT$ acts on $\cX$ fiberwise. Further, the action commutes with $\eta$, and coincides with the action on the first factor restricted to $\cX\times_{\AA^1} (\AA^1\setminus\{0\})\cong X\times (\AA^1\setminus\{0\})$.
	\end{enumerate}

\end{defn}

\begin{remark}
In the case when $\xi$ is rational, if $X$ is further $\QQ$-Gorenstein, then the polarized pair $(V, B, L)$ is actually a log Fano pair (see \cite{LWX}). 
We also remark that in \cite{LX18, LWX}, only $\QQ$-Gorenstein or special test configurations are considered. For the purpose of studying the Monge-Amp\`ere energy, we don't need such assumptions, and the functional defined in \emph{loc.cit.} is still well-defined without the assumptions on test configurations.

\end{remark}

\subsection{Berkovich analytifications}
The Berkovich analytification functor associates to each algebraic variety $X$ over $k$ an analytic space $X^{\an}$ (see \cite{Berk12}). For our purposes, we shall only consider the case when $X = \Spec R$ is affine, with $R$ a finitely generated $\KK$-algebra. Working additively, we have
\[X^{\an}  = \{\mathrm{multiplicative \ semivaluations \ on \ } R \mathrm{ \ that \ are \ trivial \ on \ } \KK\}.\]
The topology on $X^{\an}$ is the weakest one such that $v\mapsto v(f)$ is continuous for all $f\in R.$ The space $X^{\an}$ contains the set of valuations $X^{\mathrm{val}}$ as a dense subset. We will denote by $|X^{\an}|$ its underlying topological space. To each $v\in X^{\an}$, we can also associate a complete residue field $\mathscr{H}(v)$, defined as the completion of the residue field $k(\ker(v))$ with respect to the norm $|\cdot| = e^{-v(\cdot)}$, where $\ker(v) = \{f\in R: v(f) = \infty\}\in X$.  

Assume $(X =\Spec R, \TT, \xi)$ is a polarized affine variety. There is a natural $\TT(\KK)$ action on $X^{\an}$ as follows: 
if $a\in \TT(\KK)$, $v\in X^{\an}$ and $f = \sum_\alpha f_\alpha$ with $f_\alpha\in R_\alpha$, then 
\[(a\cdot v) (f) = v(\sum_\alpha a^\alpha f_\alpha).\]
\begin{lem}
	Let $X^{\an, \TT}$ denote the set of $\TT$-invariant semivaluations on $X$.Then
	\[X^{\an, \TT} = \{v\in X^{\an}: v(f) =\min_\alpha v(f_\alpha), \forall f=\sum_\alpha f_\alpha\in R\}.\]
\end{lem}

\begin{proof}
	If $v$ is a semivaluation such that $v(f) = \min_\alpha v(f_\alpha)$, where $f = \sum_\alpha f_\alpha, f_\alpha \neq 0$, then it's clear that $v$ is $\TT$-invariant. Conversely, given $f= f_{\alpha_1}+\cdots + f_{\alpha_k}\in R$, we always have $v(f)\ge \min_j v(f_{\alpha_j})$. To prove the reverse inequality, note that
	\[A = \begin{pmatrix}
		a_1^{\alpha_1}       &  \dots & a_{1}^{\alpha_k} \\
		\vdots & \ & \vdots \\
		a_{k}^{\alpha_1}       & \dots & a_{k}^{\alpha_k}
	\end{pmatrix}\]
	is invertible for generic choices of $a_1, \cdots, a_k\in \TT(\KK)$. Here $a_i^{\alpha_j}$ are written in multi-index notations. Choose $a_1, \cdots, a_k$ such that $A$ is invertible. Then each $f_{\alpha_i}$ can be written as a linear combination of $(a_j\cdot f)_{j=1}^k$: 
	\[f_{\alpha_i} = \sum_{j=1}^k \lambda_{ij} a_j\cdot f.\]
	Thus $v(f_{\alpha_i}) \geq \min_{1\le j\le k} v(a_j\cdot f) = v(f), $ for $1\le i\le k$. Hence one has $v(f) = \min_{1\le i\le k} v(f_{\alpha_i})$.
\end{proof}

We now see how Berkovich spaces give a way to understand the valuation $v_\xi$ defined earlier. By functoriality, the $\TT$ action on $X$ gives a $\TT^{\an}$ action on $X^{\an}$ via analytification, that is, a map of $\KK$-analytic spaces $\mu: \TT^{\an}\times X^{\an}\to X^{\an}$. Let $g\in \TT^{\an}$, and $v\in X^{\an}$. It's worth noting that $|\TT^{\an}\times X^{\an}|\neq |\TT^{\an}|\times |X^{\an}|$ as topological spaces, and we will denote by $\pi: |\TT^{\an}\times X^{\an}|\to |\TT^{\an}|\times |X^{\an}|$ the projection map.
A point $x$ in a $\KK$-analytic space $Z$ is \emph{universal} if for any valuation field $\LL/\KK$, the algebra $\cH(x)\widehat{\otimes}_\KK\LL$ is multiplicative. It's shown in \cite{Poineau} that when $\KK$ is algebraically closed, any point $v\in X^{\an}$ is universal. In our case, we have that any $v\in X^{\an}$ and $g\in \TT^{\an}$ are universal, and Berkovich (see \cite{Berk12}) defined $g*v$ as follows: 
Let $w$ be the point of $\pi^{-1}(g,v)$ corresponding to the norm on $\cH(g)\widehat{\otimes}_\KK \cH(v)$. Then $g*v := \mu(w)$. Thus the $*$-multiplication gives a map $|\TT^{\an}|\times |X^{\an}|\to |X^{\an}|$.

We can think of $\xi$ as a monomial valuation in $\TT^{\an}$ generating an action on $X^{\an}$. As a map of algebras, the action of $\TT^{\an}$ is given by 
\[R\to R\otimes \KK[M], f_\alpha\mapsto f_\alpha\otimes \chi^\alpha.\]
Thus $(\xi*v)(f_\alpha) = v(f_\alpha)+\la \xi, \alpha\ra$, and more generally if $f = \sum_\alpha f_\alpha$ with $f_\alpha\neq 0$, then
$$(\xi*v)(f) = \min_\alpha\{v(f_\alpha)+\la \xi, \alpha\ra\}.$$
When $X$ is reduced and irreducible, and $v$ is the trivial valuation, $\xi *v$ is the valuation $v_\xi$. 
\subsection{Forms and currents on Berkovich spaces}
In this section, we introduce some basics of forms and currents due to Chambert-Loir and Ducros that will allow us to do integration over the Berkovich analytification $X^{\an}$. We refer to \cite{CLD12} for more details.

Let $f_1, \cdots, f_p$ be invertible functions on a $\KK$-analytic space $Z$ of dimension $n(\le p)$. These functions give a map 
$$f= (f_1, \cdots, f_p): Z\to \GG_m^{p, \an}$$
 and we call this map a moment map. After composing with the tropicalization map 
$$\trop: \GG_m^{p,\an} =( \Spec \KK[T_1^{\pm}, \cdots, T_p^{\pm}])^{\an}\to \RR^p, x\mapsto (-\log|T_1(x)|, \cdots, -\log|T_p(x)|),$$ we get a map $f_{\trop}: Z\to \RR^p$. 
Lagerberg \cite{lagerberg} defined the notion of super forms and super currents on Euclidean spaces. Roughly speaking, differential forms on $Z$ are locally given by pullbacks of super forms on $\RR^p$ via $f_{\trop}$. Currents are defined as continuous linear functionals on the space of smooth compactly supported forms on $Z$. In particular, functions on $Z$ are said to be continuous (resp.\ smooth) if they are locally the pullback of continuous (resp.\ smooth) functions on the Euclidean space via some moment map. 

Any $(n, n)$-form will be supported on the $n$-dimensional faces of a polyhedron in $Z$, called the \emph{characteristic polyhedron} $\Sigma_f\subset Z$, written as $\Sigma_f^{(n)}$, and we will denote by $\Pi_f$ the image of $\Sigma_f^{(n)}$ under $f_{\trop}$. A choice of a volume form on $\RR^p$ gives naturally a volume form $\mu_{\Pi_f}$ on $\Pi_f$ and via the moment map, a volume form $\mu_{\Sigma_f^{(n)}}$  on $\Sigma_f^{(n)}$.
If $\omega= f_{\trop}^*\alpha$ for some $(n,n)$-form $\alpha$ on $\RR^p$, then 
$$\int_Z \omega := \int_Z \la \omega, \mu_{\Sigma_f^{(n)}}\ra = \int_{f_{\trop}(Z)} \la\alpha, \mu_{\Pi_f}\ra.$$

One similarly defines the boundary integral for an $(n-1, n)$ form $\omega$, denoted by $\int_Z^\partial \omega$. In this context, we have versions of Stokes' and Green's theorem:
\begin{thm}[{\cite[Th\'eor\`eme 3.12.1, 3.12.2]{CLD12}}]  Let $Z$ be a separated, locally holomorphically separated $\KK$-analytic space.
	\begin{enumerate}
		\item Let $\omega$ be an $(n-1, n)$-form on $Z$ with compact support. Then 
		\[\int_Z d'\omega = \int_Z^\partial \omega.\]
		\item Let $\alpha$ be a $(p, p)$-form and $\beta$ be a $(q,q)$-form with $p+q = n-1$, and $\mathrm{supp}(\alpha)\cap \mathrm{supp}(\beta)$ compact. Then 
		\[\int_Z (\alpha\wedge d'd''\beta- d'd''\alpha\wedge \beta) = \int_Z^\partial (\alpha\wedge d''\beta - d''\alpha\wedge \beta).\]
	\end{enumerate}
\end{thm}

We shall also need the notion of psh functions on Berkovich spaces. The following definitions and facts are taken from \cite{CLD12}.
\begin{defn} A continuous function $u$ on $Z$ is plurisubharmonic (psh) if $d'd'' u\ge 0$, i.e., $\la u, d'd''\alpha\ra \geq 0$ for all compactly supported smooth $(n-1, n-1)$-forms $\alpha$ on $Z$. 

We say $u$ is psh-approachable if any point in $Z$ has an open neighborhood $U$ on which $u$ is the uniform limit of smooth psh functions.
\end{defn}

\begin{ex} Let $u$ be a function on $Z$. Assume that for any $z\in Z$, there is an open neighborhood $U$ of $z$ and a moment map $f$ on $U$ such that $u|_U = f^* v$ for some convex function $v$ on some polyhedron $ f_{\trop}(U)$. Then $u$ is psh-approachable.

\end{ex}

\section{Filtrations and Okounkov bodies}
In this section, we develop a local analogue of Okounkov bodies. In the global case, this was done in \cite{LMOkounkovBody}, and there were also local analogues in \cite{cutkosky}. 
\subsection{Okounkov Bodies}
 Let $\mu: Y\to X$ be a log resolution of $X$ at $\fm$ and set $Y_0 := \mu^{-1}(0) = \sum_{i\in I} b_i E_i$. After possibly replacing $Y$ by further blowups at $x$, one may pick a regular system of parameters $x_1, \cdots, x_n$ for $\cO_{Y,\eta}$ with $\eta$ the generic point of $\bigcap_{i=1}^n E_i$ and $x_i$ defining $E_i$. Then by Cohen structure theorem, $\widehat{\cO_{Y,\eta}}\cong \KK \lb x_1,\cdots, x_n \rb$. This gives us a rank $n$ valuation $v= (v_1, \cdots, v_n): \cO_{Y, \eta}\setminus \{0\} \to \NN^n$ with $v_1= \text{ord}_{E_1}$ on $Y$,  
 $$v_i(f) :=\mathrm{ord}_{E_i}\left(\frac{f}{\prod_{k<i}z_k^{v_k(f)}} \biggm|_{\cap_{j<i}E_j}\right)$$
 %\text{ord}_{\cap_{j\leq i} E_j}$ on $\bigcap_{j< i} E_j$$ 
 for $2\leq i\leq n$, and $\NN^n$ equipped with the lexicographic ordering. 
 
 We will need the following Izumi type estimate later:
 \begin{lem}\label{Izumi}
     There is a constant $C>0$ such that $v_i(f)\leq C \mathrm{ord}_0(f)$ for all $f\in R$ and $1\leq i\leq n$.
 \end{lem}
 \begin{proof}
     Write $f$ as an expansion $\sum_{\beta\in \NN^n} a_\beta z^\beta$, and for $t\in \RR^n_{\geq 0}$, let $w_t$ be the valuation given by 
     \[w_t(f) = \min_{\beta: a_\beta\neq 0}\{\beta \cdot t\},\]
     where $\beta\cdot t $ is the usual inner product in $\RR^n$. With this expansion, there is some $\beta_0 = (\beta_{01}, \cdots, \beta_{0n})\in \ZZ^n_{\geq 0}$ with $a_{\beta_0}\neq 0$ such that $v(f) = \beta_0$ by construction of $v$.
     Let 
     $$C_f:=\{t\in \RR^n_{\geq 0}: t_i\geq \sum_{j>i} \beta_{0j} t_j, 1\leq i\leq n-1\}.$$
     Note that this is an $n$-dimensional polyhedral cone in $\RR^n_{\geq 0}.$
     We claim that on $C_f$, 
     $$w_t(f) = v(f)\cdot t.$$
     Indeed, given any $\beta = (\beta_1, \cdots, \beta_n)\neq \beta_0$, it suffices to show $\beta_0\cdot t\leq \beta\cdot t, \forall t\in C_f.$ By choice of $\beta_0$, 
     there is a smallest index $i$ with $\beta_{0i}<\beta_i.$ The inequality $\beta_0\cdot t\leq \beta\cdot t$ then translates to 
     \[t_i\geq \frac{1}{\beta_i-\beta_{0i}}\sum_{j>i}(\beta_{0j}-\beta_j)t_j.\]
     Now for any $t\in C_f$, 
     $$ t_i\geq \sum_{j>i} \beta_{0j}t_j\geq \frac{1}{\beta_i-\beta_{0i}}\sum_{j>i}(\beta_{0j}-\beta_j)t_j.$$
     By \cite[Theorem A]{BJIzumi}, we see that $t\mapsto w_t$ is Lipschitz on $C_f$ with Lipschitz constant $C\mathrm{ord}_0 (f)$ for some constant $C>0.$ This shows $v_i(f)\leq C\mathrm{ord}_0(f).$
 \end{proof}

Now for each $m\in \NN$, set $R_m := \bigoplus_{\la \alpha, \xi\ra \leq m} R_\alpha$. Define 
$$\Gamma_m = v(R_m) \subset \NN^n. $$
As in the construction of Okounkov bodies (see \cite{LMOkounkovBody}),  
$$\Gamma:= \{(x, m): x\in \Gamma_m, m\in \NN_+\} $$
is a subsemigroup of $\NN^{n+1}$, and we will denote by $\Sigma(\Gamma)\subset \RR^{n+1}$ the closed convex cone generated by $\Gamma$. We define the convex body of $(X= \Spec R, \xi, 0)$ by 
$$\Delta\times \{1\}:= \Sigma(\Gamma) \cap (\RR^n\times \{1\}).$$

To see that this is indeed a local version of the Okounkov body(compare \cite{LMOkounkovBody}), we need the following lemma: 

\begin{lem}\label{OkounkovBodyCondition}
	Let $\Gamma$ be as above. Then $\Gamma$ satisfies the following three properties: 
	\begin{enumerate}[(1)]
		\item $\Gamma_0 = \{0\}$,
		\item There exist finitely many $a_i\in \NN^n$ such that $(a_i, 1)$ span a subsemigroup $B\subset \NN^{n+1}$ containing $\Gamma$.
		\item The subgroup generated by $\Gamma$ in $\ZZ^{n+1}$ is $\ZZ^{n+1}$.
	\end{enumerate}
\end{lem}
\begin{proof}
	The first condition is obviously satisfied. For the second property, we first claim that there is some $b \gg 0$ such that $v_i(f) \leq mb$ for all $0\leq i\leq n$ and $0\neq f\in R_m$. Granting the claim, the vectors $(a_1, \cdots, a_n, 1), 0\leq a_i \leq b$ will span a semigroup containing $\Gamma$. The claim follows from Izumi type estimates. Lemma \ref{Izumi} shows that for any $f\in R$ and $1\leq i\leq n$, $v_i(f) \leq C_i \text{ord}_0(f)$ for some constant $C$. By \cite[Proposition 4.8]{BJIzumi}, we further have some constant $C'$ such that $\text{ord}_0(f) \leq C' v_\xi (f)$. 
	Take $b =CC'$, the claim follows.
	
	It now remains to prove (3).  Write $x_i = \frac{f_i}{g_i}$ with $f_i, g_i\in R$. Then $v(f_i)-v(g_i) = e_i, 1\leq i \leq n$ where $\{e_i\}$ denotes the standard basis for $\ZZ^n$. Since $(0,1)\in \Gamma$, we have that $\Gamma$ will generate all of $\ZZ^{n+1}$.
\end{proof}

By \cite{LMOkounkovBody} and \cite[Th\'eor\`eme 1.12]{BoucksomOkounkovBody}, we have the following immediate consequences.
\begin{thm}\label{VolumeConvergence}
	For $m\geq 1$, let $\rho_m := \frac{1}{m^n} \sum_{x\in \Gamma_m}\delta_{m^{-1}x}$ be a positive measure on $\Delta$. Then $\lim_{m\to \infty} \rho_m =\rho$ weakly, where $\rho$ denotes Lebesgue measure on $\Delta$. In particular, the limit 
	$$\vol(\Delta) = \lim_{m\to \infty}\frac{n!}{m^n} \#\Gamma_m = \lim_{m\to \infty}\frac{n!}{m^n} \dim_\CC R_m$$ exists and
	equals $\vol(\xi)$.
\end{thm}

\begin{ex} Let $X = \Spec \CC[z_1, \cdots, z_n]$ with $\TT$ the $n$-dimensional torus densely embedded in $X$ and $v= (\mathrm{ord}_{z_1}, \cdots, \mathrm{ord}_{z_n})$ on monomials. The Reeb field $\xi = (\xi_1, \cdots, \xi_n)$ is a vector in $\RR^n_{> 0}$. Then 
\[\Gamma_m = \{\alpha: \la \alpha, \xi\ra\leq m\},\]
and so
\[\Delta = \{\alpha\in \RR^n_{\ge 0}: \la \alpha, \xi\ra \leq 1\}.\]
\end{ex}

\subsection{Filtrations}
Now we look at similar constructions when we have a filtration on $R = \bigoplus_{\alpha}R_\alpha$. 

\begin{defn}
	A filtration $\cF$ on $R$ is a family $\cF^\lambda R$ of vector subspaces of $R$ such that 
	\begin{enumerate}[(a)]
		\item $\cF$ is decreasing: $\cF^\lambda R \subseteq \cF^{\lambda'} R$ if $\lambda \geq \lambda'$.
		\item $\cF$ is left continuous: $\cF^\lambda R = \cap_{\lambda' < \lambda} \cF^{\lambda'} R$ for $\lambda > 0$.
		\item $\cF$ is multiplicative: $\cF^\lambda R_\alpha \cdot \cF^{\lambda'} R_{\alpha'} \subseteq \cF^{\lambda+ \lambda'}R_{\alpha+\alpha'}$.
		\item $\cF$ is $\TT(\KK)$-invariant: $\cF^\lambda R = \bigoplus_{\alpha\in \Lambda}\cF^\lambda R_\alpha$.
		\item $\cF^0 R = R$, and for any $\alpha\in \Lambda$, $\cF^\lambda R_\alpha = 0$ for $\lambda \gg 0$.
	\end{enumerate}
\end{defn}

We define $R_m^t:= \cF^{mt} R_m$ for $m\in \NN$, and $t\in \RR^+$, and set 
$$T_m :=T_m(\cF) := \sup \{ t\geq 0:  R_m^t \neq 0\}.$$
The multiplicativity property of the filtration gives  superadditivity of $mT_m$: $mT_m +m'T_{m'}\leq (m+m') T_{m+m'}$. Thus by Fekete's lemma, 
$$T:= T(\cF) := \lim_{m\to \infty} T_m = \sup_m T_m(\cF)$$
 exists. The filtration $\cF$ is said to be linearly bounded if $T(\cF) <\infty$. Note that being linearly bounded is independent of the choice of $\xi$.

A filtration is known to correspond to a NA norm, defined as follows:
\[\chi(f) := \sup\{\lambda\in \RR: f\in \cF^\lambda R_m\}, \forall f\in R_m.\] 
and a finitely generated filtration will correspond to a norm of finite type in the following sense.

\begin{defn}\label{finitetype}
	A $\TT$-invariant NA norm $\chi$ on $R$ is said to be of \emph{finite type} if there is a finite set of indices $A\subset \Lambda$ such that for any $f\in R$, there is a representation of $f$ taking the form 
	\[ f= \sum_{i_1, \cdots, i_p} a_{i_1\cdots i_p} g_{\beta_1}^{i_1}\cdots g_{\beta_p}^{i_p},\]
	with $a_{i_1\cdots i_p} \in \KK, g_{\beta_j}\in R_{\beta_j}, \beta_j\in A$
	such that
	\[\chi(f) = \min_{i_1\cdots i_p} \{\sum_{j=1}^p i_j\chi(g_{\beta_j})\}.\]
\end{defn}

\begin{ex}
	\begin{enumerate}[(a)]
	\item The trivial norm is a norm of finite type. 
	\item	 Each test configuration gives norm of finite type since the induced $\ZZ$-filtration is finitely generated as noted in the following proposition.
	\end{enumerate}
\end{ex}

\begin{prop}\label{TCfiltration} Any test configuration $(\cX=\Spec \cR, \xi, \eta)$ for $(X, \xi)$ with $\eta$ in the Reeb cone of $\TT\times \GG_m$ induces a filtration $\cF$ on $R$ defined by 
	$$\cF^\lambda R := \bigoplus_{\alpha\in\Lambda} \{f\in R_\alpha: t^{-\lambda}\bar{f}\in \cR_\alpha\},$$
	for $m, \lambda \in \NN$, where $\bar f$ denotes the pullback of $f$ under the composition $\cX \times_{\AA^1} (\AA^1\setminus\{0\}) \cong X\times (\AA^1\setminus\{0\})\to X$, and 
	$$\cF^\lambda R_m := \cF^{\lceil \lambda \rceil} R_m, $$
	for general $\lambda \in \RR^+$. This filtration is linearly bounded, and as a $\ZZ$-filtration, it's finitely generated, i.e. the bi-graded algebra
	\[\bigoplus_{\alpha\in \Lambda}\left(\bigoplus_{\lambda\in \ZZ} t^{-\lambda}\cF^\lambda R_\alpha\right)\]
	is a finitely generated $\KK[t]$-algebra.
\end{prop}

\begin{proof} 
In view of \cite[Lemma 2.17]{LWX}, $\cF$ is indeed a finitely generated filtration on $R$, since 
\[\cR\cong \bigoplus_{\alpha\in \Lambda}\left(\bigoplus_{\lambda\in \ZZ} t^{-\lambda}\cF^\lambda R_\alpha\right)\]
is finitely generated as a $\KK[t]$-algebra.

Note that $\eta$ generates the $\GG_m$ action in $\TT\times \GG_m$, and $\TT\times \GG_m$ acts on the central fiber,
we observe that for each $m\in \NN$, $T_m(\cF)\leq \frac 1m \min\{\la \alpha, \eta\ra | \la \alpha, \xi\ra \leq m\}$. Say $\eta = (\eta_1, \cdots, \eta_r), \xi = (\xi_1, \cdots, \xi_r)$ when fixing an isomorphism $\TT^r\cong \GG_m^r$. Write $\alpha = (\alpha_1, \cdots, \alpha_r)$, and $|\alpha|:= \sum_{i=1}^r \alpha_i$. Then $|\alpha| \min_{1\leq i\leq r}\{\xi_i\} \leq \la \alpha, \xi\ra\leq m$. Thus, 
	$$T_m(\cF)\leq \frac 1m \min\{\la \alpha, \eta\ra | \la \alpha, \xi\ra \leq m\} \leq \frac 1m \max_{1\leq i\leq r}\{|\eta_i|\}  |\alpha| \leq \frac 1m Cm =  C$$ 
	for some constant $C >0$. Therefore, $T(\cF) = \sup_m T_m(\cF)\leq C$, i.e., $\cF$ is linearly bounded.
\end{proof}

\begin{remark}
In general, if $\eta$ is not in the Reeb cone, as observed in \cite[Remark 2.18]{LWX}, one can twist $\eta$ by $m\xi'$ for $m\gg 0$ and $\xi' $ a rational perturbation of $\xi$ to get a vector in the Reeb cone.
\end{remark}

\begin{prop} Let $\cF$ be a linearly bounded filtration on $R$. Then for $t< T(\cF)$, we get local Okounkov bodies $\Delta^t$ corresponding to $\{R_m^t : m\in \NN\}$ using the construction as described above. More precisely, let
 \[\Gamma_m^t:= v(R_m^t),\]
 \[\Gamma^t:= \{(x, m): x\in \Gamma_m, m\in \NN_+\}.\]
Let $\Delta^t$ be defined by \[\Delta^t\times \{1\} = \Sigma(\Gamma^t)\cap (\RR^n\times \{1\}).\] Then the conditions $(1)-(3)$ in Lemma \ref{OkounkovBodyCondition} still hold for each $R_m^t$.
\end{prop}
\begin{proof}
Observe that the first two conditions are obviously satisfied, and so it amounts to checking (3). 
	
	Let $\cD$ be the divisor on $Y$ such that $\fm \cdot \cO_Y = \cO_Y(-\cD)$, and $\cI_i$ the ideal sheaf defining $\cap_{j\leq i} E_i, 1\leq i \leq n$. Then for $M \gg 0$, $\cO_Y(-M\cD-\cI_i)$ is relatively globally generated for all $1\leq i \leq n$. This implies that there are $f_1, \cdots, f_n\in \fm^M$ such that $\{v(f_i)\}_{i=1}^n$ generate $\ZZ^n$ as a group. Say $f_i \in R_{m_i}$ for $1\leq i \leq n$. Let $m = \max \{m_i : 1\leq i\leq n$, and $N=\sup\{t>0: f_i\in R_m^t, 1\leq i \leq n\}$. Then for $t< N$, we have $v(f_1), \cdots, v(f_n)\in \Gamma_m^t$, and $((v(f_1),m), \cdots, (v(f_n), m), (0,1)$ will generate $\ZZ^{n+1}$ as a group.
	
	Now consider the case when $N\leq t < T(\cF)$. There is some $\varepsilon >0$ such that $t+\varepsilon <T$. Thus we can pick $p\gg 0$ such that $R_p^{t+\varepsilon}\neq 0$, and $p\varepsilon >mt$. Let $g_p$ be a nonzero element in $\cF^{p(t+\varepsilon)} R_p$.  Then for all $1\leq i\leq n$,
	\begin{align*}
		f_ig_p &\in R_m^{\frac 12 N}\cdot R_p^{t+\varepsilon} = \cF^{\frac 12 Nm}R_m\cdot \cF^{p(t+\varepsilon)}R_p\\
		&\subseteq \cF^{\frac 12 Nm +pt+p\varepsilon}R_{p+m} \subseteq \cF^{\frac 12 Nm +pt+mt}R_{p+m} \\
		&\subseteq \cF^{pt+mt}R_{p+m} = R^t_{p+m}.
	\end{align*}
	Hence $v(f_1g_p), \cdots, v(f_ng_p)\in \Gamma_{m+p}$ and
	$(v(f_1g_p), m+p), \cdots, (v(f_ng_p), m+p), (0, 1)$ will be a set of generators for $\ZZ^{n+1}$.
\end{proof}

\subsection{Concave transform}
A linearly bounded filtration on $R$ further induces a concave transform
$$G: \Delta \to \RR_+$$
as follows. For $t\geq 0$, $\Delta^t\subset \Delta$ will be the local Okounkov body associated to $R_m^t$, and $G(x) = \sup \{t\in \RR_+: x\in \Delta^t\}$. It's not hard to see that $G$ is concave, upper semicontinuous and takes values in $[0, T(\cF)]$. Indeed, for any $a\in [0,1]$, if $x\in \Delta^t$, and $y\in \Delta^s$ with $t\geq s$, then $ax+(1-a)y\in  \Delta^{at+(1-a)s}$ by multiplicativity of the filtration. This shows that $G$ is concave. To see upper semicontinuity, one just needs to observe that $\{G\geq t\} = \bigcap_{s\leq t}\Delta^s$ is closed. 

As in \cite{BoucksomChen, BlumJonsson}, we have a positive measure $\mu:= G_* \rho = -\frac{d}{dt} \vol(\Delta^t)$ on $\RR_+$ of mass $\vol(\xi)$ with support $[0, T(\cF)]$. 
\begin{defn}  For a linearly bounded filtration $\cF$, we define the \emph{volume} of $\cF$ to be
$$S(\cF, \xi):= S(\cF):=\frac{n!}{\vol(\xi)}\int_0^\infty \vol(\Delta^t) dt = \frac {n!}{\vol(\xi)} \int_0^\infty td\mu(t)= \frac {1}{\vol(\Delta)}\int_\Delta G d\rho.$$
\end{defn} 
\subsection{Jumping numbers}
Given a linearly bounded filtration $\cF$ on $R$, we can also define the jumping numbers on each $R_m$:
$$0\leq a_{m,1} \leq \cdots \leq a_{m, N_m} = mT_m(\cF),$$
where $N_m = \dim_\KK R_m$, via
$$ a_{m,j} = a_{m,j}(\cF) = \inf\{\lambda \in \RR_+: \codim F^\lambda R_m \geq j\}, \  1\leq j\leq N_m.$$
Define a positive measure $\mu_m$ on $\RR_+$ for each $m$ by
$$\mu_m = \frac 1 {m^n} \sum_j \delta_{\frac{a_{m,j}}{m}} = -\frac 1{m^n}\frac{d}{dt} \dim_\CC R_m^t.$$
\begin{thm}
	Let $\cF$ be a linearly bounded filtration on $R$. Then $\mu_m$ converges weakly to $\mu$ as $m\to \infty$.
\end{thm} 
\begin{proof}
	The proof goes along the same lines as in \cite[Theorem 1.11]{BoucksomChen}. Note that $\dim_\CC \cF^\lambda R_m = j$ if and only if $\lambda \in [a_{m,N_m-j}, a_{m,N_m-j+1})$. Thus we have 
	$$\frac d {dt} \dim \cF^\lambda R_m = -\sum_j \delta_{a_{m,j}}$$
	in the sense of distributions. Let $g_m(t) = \frac 1{m^n} \dim R_m^t$. By Theorem \ref{VolumeConvergence} and the Okounkov body construction, one gets
	$$\lim_{m\to \infty} g_m(t) = g(t):=\vol \Delta(R_\bullet^t),$$
	for $0\leq t < T(\cF)$. Since $g_m$ are uniformly bounded above, $g_m\to g$ in $L^1_{\text{loc}}$ by dominated convergence, and hence $-\mu_m = g_m' \to g' = -\mu$ as distributions.
\end{proof}
We record an immediate corollary. 
\begin{cor}\label{Sconv}
	Under the same assumptions as in the previous theorem, 
	$$S(\cF) = \lim_{m\to \infty} S_m(\cF), $$
	where $S_m(\cF):= \frac{n!}{\vol(\xi)}\int_0^\infty td\mu_m(t) = \frac {n!}{\vol(\xi)m^{n+1}} \sum\limits_{j=1}^{N_m} a_{m,j} = \frac{1}{mN_m}\sum\limits_{j=1}^{N_m} a_{m,j}.$
\end{cor}
\begin{remark}
If $\cF_v$ is the filtration associated to a valuation $v$, then $S(\cF_v)$ agrees with the invariant $S(v_\xi; v)$ in \cite{XuZhuang20}.
\end{remark}
We record the following property of $S$ for later use.
\begin{prop}\label{continuity}
	For any linearly bounded filtration $\cF$, $S(\cF) = S(\cF, \xi)$ is continuous with respect to $\xi$, and homogeneous of degree $-1$.
\end{prop}
\begin{proof}
	The homogeneity property follows from directly computing $S_m$ using the formula in the previous Corollary.
	Given $\xi_l\to \xi$, one can pick $\{t_l\ge 1\}$ decreasing and $t_l\to 1$ such that $t_l\xi_l\to \xi$, and $\la \alpha, t_l\xi_l-\xi\ra>0,\forall \alpha\in \Lambda$. Then one has
	\[t_l^{-1}S(\xi_l) = S(t_l\xi_l)\leq S(\xi), \forall l.\]
	This shows that 
	\[\limsup_{l\to \infty}S(\xi_l) = \limsup_{l\to \infty} t_l^{-1}S(\xi_l)\leq S(\xi).\]
	Similarly one can pick $\{s_l\leq 1\}$ increasing with limit $1$ so that $s_l\xi_l\to \xi$ and $(s_l\xi_l-\xi)(\alpha)<0, \forall \alpha$. The homogeneity property then gives
	\[\liminf_{l\to \infty} S(\xi_l)\geq S(\xi).\]
	This completes the proof of continuity.
\end{proof}

\subsection{The quasi-regular case}
When $\xi$ is quasi-regular, as discussed in section \ref{quasiregular}, $X$ is the affine cone over a polarized pair $(V, L)$, i.e. $X=\Spec R$ where $R=\bigoplus_{t\in \ZZ_{\geq 0}}H^0(V, t L)$. In what below, we scale $\xi$ to be integral. Let $\cF$ be a filtration on $R$. In \cite{BlumJonsson}, similar invariants are defined using the Okounkov bodies. We now show that the volume of $\cF$ we defined agrees with the $S$-invariant defined in \cite{BlumJonsson}. 
Given a linearly bounded filtration $\cF$ we denote by 
$$\tilde{S} = \lim\limits_{t\to \infty} \tilde{S}_t: =  \lim_{t\to \infty} \frac{1}{tN_t} \sum_j a_{t,j}$$
the $S$-invariant for $(V, L)$ defined in \cite{BlumJonsson}. Here $N_t := h^0(V, tL)$, and $\{a_{t,j}\}$ are the jumping numbers of $\cF$.
\begin{lem}\label{volEqual}
	In the notation above, 
	\begin{enumerate}[(a)]
		\item $\vol(\xi) = \vol(L).$
		\item $ S(\cF) =\frac{n}{n+1} \tilde{S}(\cF).$
	\end{enumerate}
\end{lem}

\begin{proof}
	$(a)$	Fix $\varepsilon>0$. By Theorem \ref{VolumeConvergence} we have that 
		\[\vol(\xi) = \lim_{m\to \infty} \frac{n!}{m^n}\sum_{t\leq m} N_t.\]
		Let $P(t)$ be the Hilbert polynomial of $L$. Note that $N_t = P(t)= \frac{(L^{n-1})}{(n-1)!}t^{n-1}+O(t^{n-2})$ for $t\gg 0$ by orbifold Riemann-Roch(see \cite[(2.17)]{RossThomas}. Pick $T\in \ZZ_{\geq 0}$ such that $N_t=P(t)$ for $t\geq T$.
		Now
		\begin{align*}
			\left|\frac{n!}{m^n}\sum_{t\leq m} N_t-\frac{n!}{m^n}\sum_{t\leq m} P(t)\right|&\leq \left|\frac{n!}{m^n}\sum_{t\leq T}(N_t-P(t))\right| \\
			&\leq \frac{n!(T+1)}{m^n}\sup_{t\leq T}|N_t-P(t)|< \varepsilon,
		\end{align*}
		for $m\gg 0$.
		Thus
		\begin{align*}
			\vol(\xi) &= \lim_{m\to \infty} \frac{n!}{m^n}\sum_{t\leq m} N_t = \lim_{m\to \infty} \frac{n!}{m^n}\sum_{t\leq m} P(t)\\
			& = \lim_{m\to \infty} \frac{n(L^{n-1})}{m^n}\sum_{t\leq m} (t^{n-1}+O(t^{n-2})) = \frac{n(L^{n-1})}{m^n}(\frac{m^n}{n}+O(m^{n-1}))=(L^{n-1}) = \vol(L),
		\end{align*}
		where the third to last equality follows from Faulhaber's formula.
		
		$(b)$ By Corollary \ref{Sconv} and part $(a)$ we can write
		\[S(\cF) = \lim_{m\to \infty} \frac{n!}{\vol(L)m^{n+1}}\sum_{t\leq m} tN_t\tilde{S}_t.\]
		For fixed $\varepsilon >0$, we can pick $T$ so that $|\tilde{S}-\tilde{S}_t|<\varepsilon$, and $N_t = P(t)$ for all $t \geq T$. Then 
		\begin{align*}
			&\left|\frac{n!}{\vol(L)m^{n+1}}\sum_{t\leq m}tN_t\tilde{S}_t - \frac{n!}{\vol(L)m^{n+1}}\sum_{t\leq m} tP(t)\tilde{S}\right|\\
			&\leq \left|\frac{n!}{\vol(L)m^{n+1}}\sum_{t<T}(tP(t)\tilde{S} - th^0(V,tL)\tilde{S}_t)\right|+\varepsilon \frac{n!}{\vol(L)m^{n+1}}\sum_{t=T}^mtP(t)\\
			&\lesssim \varepsilon
		\end{align*}
		for $m\gg 0$.
		Thus using again Faulhaber's formula, we get
		\begin{align*}
			S(\cF) = \lim_{m\to \infty}\frac{n!\tilde{S}}{\vol(L)m^{n+1}}\sum_{t\leq m}tP(t)=\lim_{m\to \infty}\frac{\tilde{S}}{m^{n+1}}(\frac{m^{n+1}}{n+1} +O(m^{n})) = \frac{n}{n+1} \tilde{S}.
		\end{align*}

\end{proof}

\section{Monge-Amp\`ere energy of a test configuration}
In this section, we prove that the volume of a filtration arising from a test configuration agrees with the invariant considered in \cite{LX18, LWX}.
\begin{defn}
	We define the NA Monge-Amp\`ere energy of a test configuration $(\cX, \xi, \eta) $ for $(X, \xi)$ to be 
	$$E^{\mathrm{NA}}(\cX,\xi, \eta):= S(\cF),$$
	where $(\cX,\xi, \eta)$ is a test configuration for $(X,\xi)$, and $\cF = \cF(\cX, \eta)$ is the corresponding filtration.
\end{defn} 

The following theorem addresses part of the Theorem \ref{main}.
\begin{thm}\label{MAenergy}
	Let $(\cX = \Spec\cR,\xi, \eta)$ be a test configuration for $(X,\xi)$ with $\eta$ in the Reeb cone, and let $\cF = \cF(\cX, \eta)$ be the corresponding filtration. Then
	$$E^{\mathrm{NA}}(\cX,\xi, \eta) =\frac{1}{n+1}\frac{D_{-\eta} \vol_{\cX_0}(\xi)}{\vol(\xi)}.$$
\end{thm}

\begin{proof}
	Let $\fa_k= \bigoplus_{\alpha} \{f\in R_\alpha: s^{-k}\bar{f}\in \cR_\alpha\}$, where $\bar{f}$ are defined as in Proposition \ref{TCfiltration}, for $k \in \ZZ$. By \cite[Lemma 2.17]{LWX}, this defines a graded sequence of ideals with $\fa_{k+1}\subset \fa_k$ and $\fa_k= R$ for $k\leq 0$. Further we can write $\cR = \bigoplus_{k\in \ZZ} s^{-k}\fa_k$. Then the central fiber is given by
	$$\cX_0 = \Spec \bigoplus_k (\fa_k / \fa_{k+1}) = \Spec \bigoplus_{k,\alpha}\frac{\fa_k\cap R_\alpha}{\fa_{k+1}\cap R_\alpha}.$$
	Let $w_{\alpha,k} = \dim_\KK \frac{\fa_k\cap R_\alpha}{\fa_{k+1}\cap R_\alpha}.$ 
	First note that 
	$$\vol_{\cX_0}(\xi) = \lim_{m\to \infty}\frac{n!}{m^n}\sum_{\la\alpha, \xi\ra\leq m}\sum_k w_{\alpha,k}  = \lim_{m\to \infty}\frac{n!}{m^n}\sum_{\la\alpha, \xi\ra\leq m}\dim_\KK R_\alpha = \vol(\xi).$$
	In view of \cite[Theorem 4.14]{CSIrregular}, 
	$\sum_{\alpha,k}e^{-t\la\alpha, \xi\ra}kw_{\alpha,k}$ admits a meromorphic expansion in a neighborhood of $0\in \CC$ of the form
	$$\sum_{\alpha,k}e^{-t\la\alpha, \xi\ra}kw_{\alpha,k} = D_{-\eta} \vol_{\cX_0}(\xi)t^{-n-1}+O(t^{-n}).$$
	In particular, for $m \gg 0$, 
	$$\sum_{\alpha,k}e^{-\frac{\la\alpha, \xi\ra} m}kw_{\alpha,k} = D_{-\eta} \vol_{\cX_0}(\xi)m^{n+1}+O(m^{n}).$$
	On the other hand, 
	$$S(\cF) = \lim_{m\to \infty}S_m(\cF) = \lim_{m\to \infty} \frac{n!}{m^{n+1}\vol(\xi)}\sum_{\la \alpha, \xi\ra\leq m, k}kw_{\alpha,k}.$$
	It remains to show that 
	\[\lim_{m\to \infty}\frac{1}{m^{n+1}}\sum_{\alpha,k}e^{-\frac{\la\alpha, \xi\ra}{m}}kw_{\alpha, k} = (n+1)\vol(\xi)S(\cF, \xi).\tag{1}\]
	It's shown in \cite[Lemma 2.11]{LWX}(see also \cite{CS19, LX18, MSY}) that $\vol_{\cX_0}(\xi)$ is smooth in $\xi$, and hence $D_{-\eta} \vol_{\cX_0}(\xi)$ is continuous in $\xi$.
	By Proposition \ref{continuity}, it therefore suffices to show $(1)$ when $\xi$ is rational. By the homogeneity property, we may also assume $\xi$ is integral. In this case, $X=C(V, L)$, and 
	\[R = \bigoplus_{t\in \ZZ_{\geq 0}} H^0(V, t L).\]
	Then for each $t$, we are summing up the jumping numbers for the filtration on $H^0(V, tL)$, and so
	\[\lim_{m\to \infty}\frac{1}{m^{n+1}}\sum_{\alpha,k}e^{-\frac{\la\alpha, \xi\ra}{m}}kw_{\alpha, k} = \lim_{m\to \infty}\frac{1}{m^{n+1}}\sum_{t=0}^\infty e^{-\frac{t}{m}}t N_t \tilde{S}_t(L) \]
	
	Fix $\varepsilon >0$. There is some constant $C_1>0$ such that $N_t$ equals the Hilbert polynomial  $\frac{\vol(\xi)}{(n-1)!}t^{n-1}+O(t^{n-2})$ for $t> C_1$, and some other constant $C_2(\varepsilon) >0$ such that $|\tilde{S}_t - \tilde{S}| <\varepsilon$ for $t> C_2$. Let $C = \max\{C_1, C_2\}$. Then there is some constant $M$ such that 
	\begin{align*}
		&\left|\frac{1}{m^{n+1}}\sum_{t=0}^\infty e^{-\frac{t}{m}}t N_t \tilde{S}_t -\frac{\vol(\xi)}{m^{n+1}(n-1)!}\sum_{t=0}^\infty e^{-\frac{t}{m}}t^n \tilde{S}\right| \\
		&\leq \frac{1}{m^{n+1}}M+	\frac{\vol(\xi)}{(n-1)!m^{n+1}}\left|\sum_{t=C+1}^\infty e^{-\frac{t}{m}}(t^n+O(t^{n-1}) )\tilde{S}_t -  \sum_{t=C+1}^\infty e^{-\frac{t}{m}}t^nS\right|\\
		&\leq 	\frac{1}{m^{n+1}}M+	\frac{\varepsilon\vol(\xi)}{(n-1)!m^{n+1}}\left(\sum_{t=C+1}^\infty e^{-\frac{t}{m}}t^n+ O(m^n) \right)\\
		&\leq 	\frac{1}{m^{n+1}}M+	\frac{\varepsilon\vol(\xi)}{(n-1)!m^{n+1}} (n! m^{n+1} +O(m^n))\\
		&\lesssim\varepsilon
	\end{align*}
	for $m\gg 0$.
	Thus by Lemma \ref{volEqual}
	\begin{align*}
		\lim_{m\to \infty}\frac{1}{m^{n+1}}\sum_{t=0}^\infty e^{-\frac{t}{m}}t N_t \tilde{S}_t(L) =\lim_{m\to \infty} \frac{\vol(\xi)}{m^{n+1}(n-1)!}\sum_{t=0}^\infty e^{-\frac{t}{m}}t^n \tilde{S}
		=(n+1)\vol(\xi) S.
	\end{align*}
\end{proof}

\begin{remark}\label{rationalEnergy}
This proposition together with \cite[Proposition 4.8]{LWX} shows that $E^{\na}$ computes the limit slope of the classical Monge-Amp\`ere Energy.
\end{remark}

\begin{cor}Suppose $\xi$ is integral, and $(\cX, \eta)$ is a test configuration for $(X,\xi)$. Let $\cF$ be the induced filtration, and $(V, L)$ be the polarized pair so that $X = C(V, L)$. Then $\cF$ induces a test configuration $(\cV, \cL)$ for $(V, L)$, and we have
\[E^{\na}(\cX, \xi, \eta) = \frac{1}{n+1}E^{\na}(\cV, \cL):= \frac{(\bar{\cL}^n)}{(n+1)\vol(L)}. \]
\end{cor}
\begin{proof}It follows from the previous proposition and Lemma \ref{volEqual} that 
$$E^{\na}(\cX, \eta) = \frac{n}{n+1}\tilde{S}(\cF).$$ The fact that $\tilde{S}(\cF) = E^{\na}(\cV, \cL)$ is shown in \cite{BHJ17}
\end{proof}

\section{Monge-Amp\`ere energy of Fubini-Study functions}

In this section, we prove Theorem \ref{main} and express the Monge-Amp\`ere energy of a finitely generated filtration in terms of an integral on $X^{\an}$.
\subsection{Fubini-Study functions}

\begin{defn}
	A Fubini-Study function on $X^{an}$ is a function of the form
	%	\[\varphi = \sup \{\frac{\log|f_\alpha|-\log\|f_\alpha\|}{\la \xi, \alpha\ra}: f_\alpha\in R_\alpha, \alpha\in \Gamma\setminus \{0\}\},\]
	
	\[\varphi = \max_{1\leq j\le N} \{\frac{\log|f_j|+\lambda_j}{\la \alpha_j, \xi\ra}:   f_j\in R_{\alpha_j}\}\]
	
	where $\bigcap_{i=1}^N \{f_i = 0\} = \{0\}$, and $\lambda_j\in \RR$.

We will denote by $\mathrm{FS}(\xi)$ the set of Fubini-Study functions on $X^{\an}$ with polarization $\xi$.
\end{defn}
\begin{remark}
It's worth noting that $\mathrm{FS}(\xi)$ behaves differently from its analog in the global situation(see \cite[Proposition 2.9]{BJ18}). For example, when $\xi$ is irrational, given $\varphi,\psi \in \mathrm{FS}(\xi)$, $\varphi+\psi$ is no longer Fubini-Study. It also makes no sense to add $\varphi\in \mathrm{FS}(\xi)$ and $\psi\in \mathrm{FS}(\xi')$ for $\xi, \xi'$ linearly independent.
\end{remark}

\begin{lem}
	If $\chi$ is a $\TT(\KK)$-invariant norm on $R$ of finite type, then the function induced by $\chi$ given by
	\[\varphi = \sup_\alpha\{\frac{\log|f_\alpha|+\chi(f_\alpha)}{\la \alpha, \xi\ra}: \alpha\in \Lambda \mathrm{\ and \ } f_\alpha\in R_\alpha \}. \tag{$*$}\]
	is a Fubini-Study function.
	In particular, when $\chi$ is the trivial norm, we denote by $\varphi_\xi$ the corresponding Fubini-Study function 
	\[\varphi_\xi = \sup_{\alpha\in\Lambda}\{\frac{\log|f_\alpha|}{\la \alpha, \xi\ra}: f_\alpha\in R_\alpha\} \]
\end{lem}

\begin{proof}
	It suffices to show that we can take finitely  many $(f_\alpha)$ so that the sup in $(*)$ is a max.
	Let $A$ be a finite set of indices that determines $\chi$ as in Definition \ref{finitetype}.

	For each $f\in R_\beta$, by definition we can put $f$ into the form $f= \sum_{i_1, \cdots, i_p} a_{i_1\cdots i_p} g_{\beta_1}^{i_1}\cdots g_{\beta_p}^{i_p}$, with $\beta_j \in A$. Then 
	\[\beta = i_1\beta_1+\cdots i_p\beta_p,  \ \ \ \chi(f) = \min \sum_j i_j\chi(g_{\beta_j}), \]
	and we can further assume these $g_{\beta_j}$ form a basis for $\bigoplus_{\alpha\in A}R_\alpha$. Thus,
	\[\frac{\log|f|+\chi(f)}{\la \beta, \xi\ra} \leq \max_{i_1, \cdots, i_p}\{\frac{\log|g_{\beta_1}^{i_1}\cdots g_{\beta_p}^{i_p}|+\sum_ji_j\chi(g_{\beta_j})}{\la \beta, \xi\ra}\} \leq \max_{\beta_j}\{\frac{\log|g_{\beta_j}|+\chi(g_{\beta_j})}{\la \beta_j, \xi\ra}\},\]
	we have
	\[\sup_{\alpha\in\Lambda} \{\frac{\log|f_\alpha|+\chi(f_\alpha)}{\la \alpha, \xi\ra}\} = \max_{\beta\in A}\{\frac{\log|f_\beta|+\chi(f_\beta)}{\la \beta, \xi\ra}\}.\]
	
	By adding finitely many $(f_\alpha)_{\alpha\in A'}$ so that $(f_\beta)_{\beta\in A\cup A'}$ generate a $\fm$-primary ideal, we can write 
	\[\varphi = \max_{\alpha\in A\cup A'} \{\frac{\log|f_\alpha|+\chi(f_\alpha)}{\la \alpha, \xi\ra}\}. \]
\end{proof}

In what below, we follow the notation used in \cite{CLD12}.
\begin{prop}
	Let $\varphi$ be a Fubini-Study function. Then $\varphi|_{X^{\an}\setminus \{0\}}$ is a continuous $\xi$-equivariant psh-approachable function, that is, it's a continuous psh-approachable function, with the property that
	\[\varphi((t\xi)*v) = \varphi(v) -t, \forall t\in \RR, v\in X^{\an}.\]
\end{prop}

\begin{proof}
	The proof goes along the same lines as in \cite[Proposition 6.3.2]{CLD12}. Let
	\[\varphi = \max{\left\{\frac{\log|f_{1}| +\lambda_1}{\la \alpha_1, \xi\ra}, \cdots, \frac{\log|f_{p}|+\lambda_p}{\la \alpha_p,\xi\ra}\right\}}.\]
	
	Given $x \in X^{\an}\setminus\{0\}$, assume $f_1(x), \cdots, f_r(x)\neq 0$ and $f_{r+1}(x) = \cdots = f_p(x) = 0$ at $x$. Then after passing to the neighborhood 
	\[U:= \left\{x: \min_{1\leq i\leq r}{\{\log|f_i(x)|\} > \max_{r+1\leq j\leq p}{\{x: \log|f_j(x)|\}}}\right\},\]
	$$\varphi|_U = \max\left\{\frac{\log|f_{1}| +\lambda_1}{\la \alpha_1, \xi\ra}, \cdots, \frac{\log|f_{r}|+\lambda_r}{\la \alpha_r, \xi\ra}\right\}.$$
	Now $f_1, \cdots, f_r$ are invertible functions on $U$, and hence give a moment map $U\to \GG_m^r$ by $v\mapsto (f_1(v), \cdots, f_r(v))$. Since $\max(x_1, \cdots, x_r)$ is a continuous and convex function on $\RR^r$, this shows that $\varphi$ is psh-approachable.
	The $\xi$-equivariance follows from the fact that $(t\xi* v) (f) = \min\limits_{\alpha: f_\alpha\neq 0} \{v(f_\alpha)+ t\la \alpha, \xi\ra\}, $ where $f = \sum_\alpha f_\alpha$.
	
\end{proof}

\subsection{Monge-Amp\`ere measures}
In view of the previous proposition, given Fubini-Study functions $\varphi_1, \cdots, \varphi_n$, they are psh-approachable and hence define a positive $(n, n)$-current $d'd''\varphi_1\wedge \cdots\wedge d'd''\varphi_n$ in the Bedford-Taylor sense.

\begin{prop}\label{SuppMA}
Let $\varphi_1, \cdots, \varphi_n$ be Fubini-Study functions. Then  
 $$d'd''\varphi_1 \wedge\cdots \wedge  d'd''\varphi_n = 0$$
  on $X^{\an}\setminus\{0\}$.
	\end{prop}
For the proof of this proposition, we need the following result from \cite[Proposition 6.3.2]{CLD12}(see also \cite[Lemme 5.18]{Demailly}) which gives a regularized max function:
\begin{lem}\label{SmApprox}
	Let $p$ be a positive integer, and $\eta>0$ a positive real number. There is a function $M_\eta: \RR^p \to \RR$ satisfying the following properties:
	
	\begin{enumerate}
		\item $M_\eta$ is smooth, convex, and increasing in each variable;
		\item For all $(x_1, \cdots, x_p)\in\RR^p$,
		\[\max(x_1, \cdots, x_p) \le M_\eta(x_1, \cdots, x_p)\le \max(x_1,\cdots, x_p)+\eta;\]
		\item If $x_i+\eta \le \max_{j\neq i} (x_j-\eta)$, then 
		\[M_\eta(x_1, \cdots, x_p) = M_\eta(x_1, \cdots,\hat{x}_i, \cdots, x_p);\]
		\item For all $t\in \RR$, 
		\[M_\eta(x_1+t, \cdots, x_p+t) = M_\eta(x_1, \cdots, x_p)+t.\]
	\end{enumerate}
\end{lem}

\begin{proof}[Proof of Proposition \ref{SuppMA}]
   We first show that $(d'd''\varphi)^n = 0$ for one Fubini-Study function $\varphi$.
    
    Let $x\in X^{\an}\setminus\{0\}$. As observed in the previous proposition, we can assume that $\varphi$ takes the form 
	\[ \varphi = \max \left\{\frac{\log|f_{1}|+\lambda_1}{\la \alpha_1, \xi\ra}, \cdots, \frac{\log|f_{p}|+\lambda_p}{\la \alpha_p, \xi\ra}\right\} = f_{\trop}^*\psi\]
	in some neighborhood $U$ of $x$ in $X^{\an}\setminus\{0\}$, where $f = (f_1, \cdots, f_p): U\to \GG_m^p$ gives a moment map, $\psi: \RR^p\to \RR$ is given by $\psi(x_1, \cdots, x_p) = \max \left\{\frac{-x_1+\lambda_1}{\la \alpha_1, \xi\ra}, \cdots, \frac{-x_p+\lambda_p}{\la \alpha_p, \xi\ra}\right\}.$ In view of the lemma above, we have a sequence of smooth psh functions $\varphi_\varepsilon = f_{\trop}^*\psi_\varepsilon$, where $\psi_\varepsilon(x_1, \cdots, x_p) := M_\varepsilon(\frac{-x_1+\lambda_1}{\la \alpha_1, \xi\ra}, \cdots, \frac{-x_p+\lambda_p}{\la \alpha_p, \xi\ra})$, converging uniformly to $\varphi$. Moreover, $\varphi_\varepsilon$ preserves $\xi-$equivariance in the sense that 
	\begin{align*}
	\varphi_\varepsilon((t\xi)* v) &= f_{\trop}^*\psi_\varepsilon(x_1+t\la \alpha_1, \xi\ra, \cdots, x_p+t\la \alpha_p, \xi\ra) \\
	&= f_{\trop}^*\psi_\varepsilon(x_1, \cdots, x_p)-t = \varphi_\varepsilon(v)-t.
	\end{align*}
	By \cite[Corollary 5.6.5]{CLD12}, we have $(d'd''\varphi)^n = \lim\limits_{\varepsilon\to 0}(d'd''\varphi_\varepsilon)^n$ in $U$.
	
	Let $y$ be any point in $U$. If the tropical dimension of $y$ is less than $n$, then there is a compact analytic neighborhood $V\subset U$ of $y$ such that $f_{\trop}(V)$ has dimension less than $n$. In this case, $(d'd''\varphi_\varepsilon)^n = 0$ on $V, \forall \varepsilon >0$. So we may assume that $y\in \Sigma_f^{(n)}$, and in particular, we will assume $p\geq n$.

	Let $u$ be a compactly supported smooth function on $U$ with support contained in the union of $n$-cells containing $x$. Now for a convenient cell decomposition $\cC$ of $\Sigma_f^{(n)}$, we have
	\[\int_U u(d'd''\varphi)^n = \lim_{\varepsilon\to 0+}\int_U u (d'd''\varphi_\varepsilon)^n = \lim_{\varepsilon\to 0+}\sum_{C\in \cC \ \dim C = n \ x\in C}\int_C u (d'd''\varphi_\varepsilon)^n\]
	
	For a fixed cell $C\in \cC$ containing $x$, if $y\in \inte(C)$, then $(t\xi)*y\in \inte(C)$ for $0<t\ll 1$ as well since $\inte(C)$ is open in $\Sigma_f$. Set $D = f_{\trop} (C)$. Then $f_{\trop}$ is a homeomorphism between $C$ and $D$. Let $\tilde{u}$ be the function on $D$ such that $u|_C = f_{\trop}^*(\tilde{u})$. Then 
	\[\int_C u(d'd''\varphi_\varepsilon)^n = \int_D \tilde{u}\la (d'd''\psi_\varepsilon)^n, \mu_D\ra,\]
	with $\psi_\varepsilon$ satisfying $\psi_\varepsilon(x_1+t\la\alpha_1, \xi\ra, \cdots, x_p+t\la \alpha_p, \xi\ra) = \psi_\varepsilon(x_1, \cdots, x_p)-t$ on $\inte(D)$ for $0<t\ll 1$. Since $\dim D=n$, and $\psi_\varepsilon$ is linear in one direction, the Hessian of $\psi_\varepsilon$ has to be zero. Thus we can conclude that $(d'd''\varphi)^n =0$ in a neighborhood of $x$. 
	
	Now for different Fubini-Study functions $\varphi_1, \cdots, \varphi_n$, we may assume every $\varphi_i$ is of the form 
	\[ \varphi_i = \max \left\{\frac{\log|f_{1}|+\lambda_{1i}}{\la \alpha_1, \xi\ra}, \cdots, \frac{\log|f_{p}|+\lambda_{pi}}{\la \alpha_p, \xi\ra}\right\},\]
	by possibly adding some terms of the form $\frac{\log|f_{j}|+\lambda_{ji}}{\la \alpha_j, \xi\ra}$ with $\lambda_{ji}\ll 0$. Then we proceed as before repeating the same argument.

\end{proof}

From now on, we will write $\varphi_\xi^+ = \max\{\varphi_\xi, 0\}$. 

\begin{cor}\label{measureSupp}
	For any Fubini-Study function $\varphi_1,\cdots, \varphi_{n-1}$, the measure 
	$$d'd''\varphi_1\wedge \cdots \wedge d'd''\varphi_{n-1}\wedge d'd''\varphi_\xi^+$$
	is a measure with finite support on $\TT$-invariant points of $\{\varphi_\xi = 0\}\subset X^{\an}$.
\end{cor}
We first record a lemma that will be needed in the proof.
\begin{lem}
    Let $\psi: \RR^k\to \RR$ be a function of the form $\psi = \max\{v_1\cdot x, \cdots, v_l\cdot x\}$, where $v_1, \cdots, v_l $ are nonzero vectors in $\RR^k$, and $\cdot $ denotes the usual inner product in $\RR^k$. Let $\psi_n$ be a sequence of smooth convex functions converging locally uniformly to $\psi$. Let $V_i = \{v_i\cdot x=0\}$. Then 
    \[\sum_{i=1}^l [V_i]\wedge d''(-v_i\cdot x) = \lim_{n\to \infty} [\psi = 0]\wedge d''(-\psi_n)\]
    as $(1,1)$-supercurrents of on $\RR^k$.
\end{lem}
\begin{proof}
Let $\alpha$ be a smooth compactly supported $(n-1,n-1)$-form on $\RR^k.$ 
Then 
\begin{align*}
   \la d'd''\max{\{0,\psi\}}, \alpha\ra &= \int_{\RR^k} \max{\{0,\psi\}} d'd''\alpha = \int_{\{\psi\geq 0\}} \psi d'd''\alpha = \lim_{n\to \infty} \int_{\{\psi\geq 0\}} \psi_n d'd''\alpha\\
   &=\lim_{n\to \infty} \left(\int_{\{\psi\geq 0\}} \alpha d'd''\psi_n+ \int_{\{\psi\geq 0\}}^{\partial} (\psi_n d''\alpha - \alpha d''\psi_n)\right)\\
   &= \lim_{n\to \infty} \left(\int_{\{\psi\geq 0\}} \alpha d'd''\psi_n+ \int_{\{\psi\geq 0\}}^{\partial}  - \alpha d''\psi_n\right)
\end{align*}
By \cite[Proposition 4.12]{lagerberg}, the left hand side 
\[\la\sum_{i=1}^l [V_i]\wedge d''(-v_i\cdot x), \alpha\ra = \la d'd''\max{\{0,\psi\}}, \alpha\ra - \lim_{n\to \infty} \int_{\{\psi\geq 0\}} \alpha d'd''\psi_n.\]
This proves the lemma.
\end{proof}

\begin{proof}[Proof of Corollary \ref{measureSupp}]
    We again first deal with the case when $\varphi = \varphi_1 = \cdots  = \varphi_{n-1}.$
	Note that $\varphi_\xi^+$ is psh-approachable, and so $(d'd''\varphi)^{n-1}\wedge d'd''\varphi_\xi^+$ indeed defines a measure on $X^{\an}$ by \cite[Corollary 5.6.6]{CLD12}. By the previous proposition, we have $(d'd''\varphi)^{n-1}\wedge d'd''\varphi_\xi = 0$. Thus $(d'd''\varphi)^{n-1}\wedge d'd''\varphi_\xi^+$ is zero on $\{\varphi_\xi\neq 0\}$.
	Let $x\in \{\varphi_\xi = 0\}$ and let $U$ be an open neighborhood of $x$ in $X^{\an}$ on which $\varphi$ and $\varphi_\xi$ are uniformly approximated by smooth $\xi$-equivariant psh functions $\varphi_\varepsilon$ and $\varphi_{\xi, \varepsilon}$. Let $V\subset U$ be a compact analytic domain containing $x$ such that $f_{\trop}(V)$ is a PL space of dimension $n$. Let $u$ be a smooth function compactly supported on $V$.
	Then
	\begin{align*}
	    \la (d'd''\varphi_\varepsilon)^{n-1}\wedge d'd''\varphi^+_{\xi} , u\ra 
		&=\la d'd''\varphi^+_{\xi}, u(d'd''\varphi_\varepsilon)^{n-1}\ra 
		= \int_U \varphi^+_{\xi} d'd''(u(d'd''\varphi_\varepsilon)^{n-1})\\
		&= \int_{U\cap \{\varphi_{\xi} \ge 0\}} \varphi_{\xi} d'd''(u(d'd''\varphi_\varepsilon)^{n-1})\\
	\end{align*}
	By Green's theorem %and lemme 7.7.3,
	\begin{align*}
		&\int_{U\cap \{\varphi_{\xi} \ge 0\}} \varphi_{\xi, \varepsilon'} d'd''(u(d'd''\varphi_\varepsilon)^{n-1})
		= \int_{U\cap \{\varphi_{\xi} \ge 0\}} u(d'd''\varphi_\varepsilon)^{n-1}\wedge d'd''\varphi_{\xi, \varepsilon'}\\
		&+ \int_{U\cap \{\varphi_{\xi} \ge 0\}}^{\partial} \left( \varphi_{\xi, \varepsilon'} d''(u(d'd''\varphi_\varepsilon)^{n-1}) - d''\varphi_{\xi,\varepsilon'
		}\wedge (u(d'd''\varphi_\varepsilon)^{n-1} \right)\\
		%&= \int_{U\cap \{\varphi_{\xi, \varepsilon} \ge 0\}}^{\partial} \left(\varphi_{\xi, \varepsilon} d''(u(d'd''\varphi_\varepsilon)^{n-1}) - d''\varphi_{\xi, \varepsilon}\wedge (u(d'd''\varphi_\varepsilon)^{n-1})\right)
	\end{align*}
	where the last equality follows from Green's theorem. Further, note that the term 
	$$\int_{U\cap \{\varphi_{\xi} \ge 0\}} u(d'd''\varphi_\varepsilon)^{n-1}\wedge d'd''\varphi_{\xi, \varepsilon'}\to 0$$  and 
	\[\int_{U\cap \{\varphi_{\xi} \ge 0\}}^{\partial}  \varphi_{\xi, \varepsilon'} d''(u(d'd''\varphi_\varepsilon)^{n-1})\to 0\]
	as $\varepsilon', \varepsilon\to 0$
	
	As in the proof of the previous proposition, we may assume that $x\in \Sigma_f^{(n)}$, and fix a convenient cell decomposition $\cC$ for $\Sigma_f$. Then
	\begin{align*}
		&\int_{U\cap \{\varphi_{\xi} \ge 0\}}^{\partial}  - d''\varphi_{\xi, \varepsilon'}\wedge (u(d'd''\varphi_\varepsilon)^{n-1})\\
		&=\sum_{\substack{ C\in\cC \\  \dim(C)=n}} \sum_{\substack{x\in F\subset C \\ \dim(F) = n-1 \\ F\in \cC}} \int_{F\cap U\cap \{\varphi_{\xi}\ge 0\}} - d''\varphi_{\xi, \varepsilon'}\wedge (u(d'd''\varphi_\varepsilon)^{n-1})
	\end{align*}
	Fix a cell $x\in F\subset \partial C\in \cC$. Let $D = f_{\trop}(C)$ and $G = f_{\trop}(F)\subset \partial D$. Now on $F$, let $\psi_\varepsilon, \psi_{\xi, \varepsilon'}, v$ be functions on $D$ which pull back to functions $\varphi_\varepsilon, \varphi_{\xi,\varepsilon'}, u$ respectively.
	
	\begin{align*}
		&\int_{F\cap U\cap \{\varphi_{\xi}\ge 0\}}  -d''\varphi_{\xi, \varepsilon'}\wedge (u(d'd''\varphi_\varepsilon)^{n-1}) 
		= \int_G  \la - d''\psi_{\xi, \varepsilon'}\wedge (v(d'd''\psi_\varepsilon)^{n-1}), \mu_D\ra|_{(G, D)}\\
		%&= \int_G v \la (d'd''\psi_{ \varepsilon})^{n-1}\wedge d''\psi_{\xi, \varepsilon}, \mu_D\ra|_{(G, D)}
	\end{align*}
	
	We may assume that the cell decomposition is fine enough so that $G\subset \{x_j=0\}$ for some $j$ and that by the previous lemma
    \begin{align*}
	\int_{G\subseteq \partial D}  \la - d''\psi_{\xi, \varepsilon'}\wedge (v(d'd''\psi_\varepsilon)^{n-1}), \mu_D\ra|_{(G, D)}
	&\to \int_{G\subseteq \partial D} v(d'd''\psi_\varepsilon)^{n-1}\wedge d''x_j\\
	&= \int_G v (d'd''(\psi_\varepsilon|_{\{x_j=0\}})^{n-1}
	\end{align*}
	as $\varepsilon'\to 0$. 
	Let $\varepsilon \to 0$, in view of \cite[Lemme 5.7.4]{CLD12}, the last integral converges to 
	$$\sum_{y\in S_x}\lambda_y u(y)$$
	where $S_x$ is a finite set of points in $G$, and $\lambda_y$ depend only on the tropical cone at $x$. This shows that the measure has discrete support.
	
	In the general case, we proceed with the same argument, and after taking $\varepsilon'\to 0$, locally we get an integral of the form 
	\[\int_{G\subseteq \partial D} v d'd''\psi_{1,\varepsilon}\wedge \cdots \wedge d'd''\psi_{n-1,\varepsilon} \wedge d''x_j.\]
	Now for each $(t_1, \cdots, t_{n-1})\in \RR^{n-1}_{\geq 0}$, we have that 
	$$ (d'd''(t_1\psi_{1,\varepsilon}+\cdots +t_{n-1}\psi_{n-1,\varepsilon}))^{n-1}\wedge d''x_j$$
	converges to a positive measure on $G$ with finite support $S_{x, t_1, \cdots, t_{n-1}}$. This shows each term in the expansion converges to a positive measure on $G$ with finite support. 
	In particular, 
	$$ d'd''\psi_{1,\varepsilon}\wedge \cdots \wedge d'd''\psi_{n-1,\varepsilon} \wedge d''x_j$$
	converges to a positive measure with finite support on $G.$
	
	Note further that one can embed $X$ into $\AA^N$ using a finite number of function $(f_\alpha)$ in $R$. For such a fixed embedding, \[\{\varphi_{\xi}\le 0\}\subset \bigcap_\alpha \{\log|f_\alpha|\leq 0\},\]
	and $\bigcap_\alpha \{\log|f_\alpha|\leq 0\}$ defines a compact disk in $\AA^{N, \an}$. Thus $\{\varphi_\xi = 0\}$ is a compact subset of $X^{\an}$, and the support has to be finite.
	
	Finally, since $\TT(\KK)$ is divisible, and each point $v$ in the support of the measure has finite $\TT$ orbit, we conclude that $v$ has to be $\TT(\KK)$-invariant.
\end{proof}

\begin{defn}
	Given a Fubini-Study function $\varphi$ on $X^{\an}$, we define the Monge-Amp\`ere energy of $\varphi$ to be
	\[E^{\na}(\varphi)  := E^{\na}(\varphi, \varphi_\xi):=\frac {1}{n\vol(\xi)} \sum_{j=0}^{n-1}\int_{X^{\an}\setminus\{0\}} (\varphi-\varphi_\xi) (d'd''\varphi)^j\wedge (d'd''\varphi_\xi)^{n-j-1}\wedge d'd''\varphi_\xi^+.\]
\end{defn}
\begin{remark}\label{FormulaEqual}
We note that this formula is identical to its counterpart in complex geometry.
Indeed let $r$ be a reference metric on $X$ viewed as a complex manifold compatible with $\xi$, i.e., $\xi = J(r\frac{\partial}{\partial r})$ where $J$ is the complex structure on $X$.  Let $\phi$ be a transverse psh potential on $X$, meaning $\phi$ is a function such that  the Lie derivatives $L_\xi \phi = L_{J\xi}\phi = 0$ and that $re^\phi$ is a psh function on $X$. Further denote the link $ \{r=1\}\cap X$ by $Y$.  Then $E(\phi)$ is given by (see \cite{CS19, LWX})
\begin{align*}
    E(\phi) &= \frac{1}{\vol(\xi)}\sum_{j=0}^{n-1}\int_{Y} \phi (dd^c\log r+dd^c\phi)^j\wedge (dd^c\log r)^{n-j-1}\wedge d^c\log r \\ 
    &= \frac{1}{\vol(\xi)}\sum_{j=0}^{n-1}\int_{X} \phi (dd^c(\log r+\phi))^j\wedge (dd^c\log r)^{n-j-1}\wedge dd^c (\max\{\log r, 0\})
\end{align*}
Here $\log r$ plays the role of $\varphi_\xi$ and $\phi$ corresponds to $\varphi-\varphi_\xi$.

\end{remark}

\subsection{Monge-Amp\`ere energy in the quasi-regular case}
We now consider the case of $X$ coming from a polarized pair $(V, L)$. This is when $\xi$ is rational. As explained in Section \ref{quasiregular}, assume $\xi \in \frac 1l \NN$, and write 
\[R = \bigoplus_{k}\left(\bigoplus_{\la \alpha, \xi \ra = \frac kl} R_\alpha\right) = \bigoplus_k H^0(V, kL).\]

A semivaluation $v$ on $V$ can also be evaluated on sections of $L$. This is done by trivializing $L$ at the center of $v$, and evaluating $v$ at the local function corresponding to the section under the trivialization.

\begin{lem}
	The map $i: V^{\an} \hookrightarrow X_\xi \subset  X^{\an}$,where $X_\xi$ is the set of $\la \xi\ra$-invariant semivaluations in $X^{\an}$, defined by sending a seimvaluation $v$ to the unique $\la\xi\ra$-invariant semivaluation such that $w(f_k) = v(f_k), \forall f_k \in H^0(V, kL)$, is an embedding of $V^{\an}$ into $X^{\an}$.
\end{lem}
\begin{proof}
	After passing to a multiple of $L$, assume $L$ is globally generated. Take any semivaluation $v\in V^{\an}$. Let $w = i(v)$ be the image. Then we have $w\in \{\varphi_\xi = 0\}$. The analytification of the map $p: X\setminus\{0\}\to V$ gives the map in the opposite direction when restricting to $\{\varphi_\xi =0\}^{\la \xi\ra}$. It's easy to check that the maps are inverses to each other. 
	
	The topology on $V^{\an}$ is the weakest one such that $v\mapsto v(f_k)$ is continuous for all $f_k \in H^0(V, kL)$ and all $k$. The topology on $\{\varphi_\xi = 0\}^{\la \xi\ra}$ induced by $p$ is the weakest such that $w\mapsto w(f_k)$ is continuous for all $f_k\in \bigoplus\limits_{\la \xi, \alpha\ra  = k/l}R_\alpha$. Since $w$ is $\la \xi\ra$-invariant, this coincides with the subspace topology inherited from $X^{\an}$. Thus $i$ gives an embedding.
\end{proof}

Let  $\varphi = \max\{\frac{\log|f_k|+\lambda_k}{\la\alpha_k, \xi\ra}: f_k\in R_\alpha\}$ be a Fubini-Study function with $\chi$ of finite type. We will also normalize $\xi$ so that it has integral weights from now on.
\begin{lem}
	Under the assumption above, $\varphi = \frac 1m \max\limits_{0\le j\le \ell} \{\log|f_j|+\lambda_j\}$ for $f_j\in H^0(V, mL)$ and some $m$ sufficiently divisible.
\end{lem}
\begin{proof}
	It suffices to show that $\varphi = \frac 1m \max\{\log|f_j|+\lambda_j: f\in R_\alpha, \la \alpha, \xi\ra = m\}$ for $m$ sufficiently divisible. Indeed, since the norm $\chi$ is of finite type, 
	$\varphi = \max\{\frac{\log|f|+\chi(f)}{\la\alpha, \xi\ra}: f\in R_\alpha, \la \alpha,\xi\ra \le M\}$. Take $m = M!$, and we are done.
\end{proof}

Thus, in view of the lemma, $\varphi$ corresponds to a $\TT$-invariant Fubini-Study metric of $L^{\an}$(see e.g. \cite{BJ18} for Fubini-Study metrics), i.e. a metric $\phi$ of the form 
$\phi = \frac 1m \max\limits_{0\le j\le \ell} \{\log|f_j|+\lambda_j\}$ where $f_j\in R_{\alpha_j}$ with $\la \alpha_j, \xi\ra = m$ are global sections of $mL$ without common zero. By introducing the trivial metric $\phi_{\mathrm{triv}}$, the metric $\phi$ can also be viewed as a function on $V^{\an}$ via $\phi - \phi_{\mathrm{triv}}$. In particular, $\varphi_\xi$ corresponds to $\phi_{\mathrm{triv}}$.

\begin{prop}
	Let $\varphi_1, \cdots, \varphi_{n-1}$ be Fubini-Study functions on $X^{\an}$ as described above, and $\phi_1, \cdots, \phi_{n-1}$ the corresponding Fubini-Study metrics on $L^{\an}$. 
	Then, under the embedding $i$ in the previous lemma, 
	\[i_*(d'd''\phi_1\wedge \cdots \wedge d'd''\phi_{n-1}) = d'd''\varphi_1\wedge\cdots \wedge d'd''\varphi_{n-1}\wedge d'd''\varphi_\xi^+,\]
where $d'd''\phi_1\wedge \cdots \wedge d'd''\phi_{n-1}$ is defined as in \cite[Sections 6.2-6.4]{CLD12}.
\end{prop}

\begin{proof}
	We again do the case when $\phi_i=\phi = \frac 1m \max\limits_{0\le j\le \ell} \{\log|s_j|+\lambda_j\}, \forall i$, where $s_j\in H^0(V, mL)$, and hence $\varphi_i = \varphi = \frac 1m \max\limits_{0\le j\le \ell} \{\log|s_j|+\lambda_j\}, \forall i$. The general case follows from the same argument. Let $x\in X_\xi$. Let $U$ be an open neighborhood of $x$ contained in $X_\xi$. By possibly shrinking $U$, we can also assume that $U\subset \{\varphi_\xi = \log|s_0|\}\subset \{\varphi_\xi = 0\}$, that $(U, s_0)$ trivializes the line bundle on $V$, and that $s_j, 0\le j\le \ell$ are nonvanishing sections on $U$. Now $\frac{s_j}{s_0}, 1\le j\le \ell$ are nonvanishing functions on $U\subset V^{\an}$ and thus give rise to a moment map $\nu=(\frac{s_1}{s_0}, \frac{s_2}{s_0}, \cdots, \frac{s_\ell}{s_0}): U\to \GG_m^\ell$. Now we can also identify $p^{-1}(U)$ with $U\times \GG_m$ via $s_0$, and thus gives a moment map $\mu=(s_0, s_1, \cdots, s_\ell): p^{-1}(U)\to \GG_m^{\ell+1}$. Thus we have the following commutative diagram of tropicalizations of moment maps:
	\[\begin{tikzcd}		p^{-1}(U) \arrow[d] \arrow[r] \arrow[r, "\mu_{\textrm{trop}}"] & \mathbb{R}^{\ell+1} \arrow[d, "\pi"] \\
		U \arrow[r, "\nu_{\textrm{trop}}"]             & \mathbb{R}^{\ell}           
	\end{tikzcd},\]
	where $\pi: \RR^{\ell+1}\to \RR^{\ell}, (z_0, \cdots, z_\ell)\mapsto (z_1-z_0, z_2-z_0, \cdots, z_\ell-z_0)$. Under these moment maps, 
	$\phi = \nu_{\trop}^*(\tilde{\phi}(y)$ where $\tilde{\phi}(y) = \frac 1m\max\{\lambda_0, -y_1+\lambda_1, \cdots, -y_\ell+\lambda_\ell\}, $ and
	$\varphi = \mu_{\trop}^*(\frac 1m\max\{-z_0+\lambda_0, \cdots, -z_\ell+\lambda_\ell\}) = \mu_{\trop}^*(\pi^*\tilde{\phi}(z_0, \cdots, z_\ell)-z_0)$. 
	
	Fix a convenient cell decomposition on $p^{-1}(U)\cap \{\varphi_\xi = 0\}$, and a cell $C$ containing $x$. On $p^{-1}(U)$, let $u$ be a compactly supported tropically smooth function with support contained in the union of $n$-cells containing $x$. After possibly replacing $\varphi, \varphi_\xi$ by smooth approximations and taking limits and applying Green's theorem, as in the proof of Corollary \ref{measureSupp}, we have 
	\begin{align*}
		\int_{p^{-1}(U)} u(d'd''\varphi)^{n-1}\wedge d'd''\varphi_\xi^+ &= \la d'd''\varphi_\xi^+, u(d'd''\varphi)^{n-1}\ra \\
		&= \int_{\{\varphi_\xi\ge 0\}\cap p^{-1}(U)}^{\partial} u (d'd''\varphi)^{n-1}\wedge d''\varphi_\xi.
	\end{align*}
	
	 Let $D$ be a facet of $C$ containing $x$. Then $\pi$ restricts to a homeomorphism (with the same calibration) of cells  $\mu_{\trop}(D)\to\nu_{\trop}(D), (0, z_1, \cdots, z_\ell)\mapsto (z_1, \cdots, z_\ell)$, and locally the last integral is given by 
	\begin{align*}
		\int_{D} u (d'd''\varphi)^{n-1}\wedge d''\varphi_\xi
	\end{align*}
	where $D$ is equipped with the induced calibration on $C$. Since $\mu_{\trop}$ and $\nu_{\trop}$ are homeomorphsims on $C$ and $D$, one can find $\tilde{u}$ and $\tilde{u}'$ such that $u|_D=\mu_{\trop}^* \tilde{u} = p^*\nu_{\trop}^*\tilde{u}'$, where $\tilde{u}' = (\pi^{-1})^*(\tilde{u})$. Then
	\begin{align*}
		\int_{D} u (d'd''\varphi)^{n-1}\wedge d''\varphi_\xi 
		&= \int_{\mu_{\trop}(D)} \tilde{u} (d'd''(\pi^*\tilde{\phi}))^{n-1}\wedge d''(-z_0)\\
		&= \int_{\nu_{\trop}(D)} \tilde{u}' (d'd''\tilde{\phi})^{n-1} %(pushforward by pi^{-1})
		= \int_D u (d'd''\phi)^{n-1}
	\end{align*}
\end{proof}

We can now conclude that $E^{\na}(\varphi)$ agrees with the Monge-Amp\`ere energy of $\phi$ considered in \cite{BHJ17}.
\begin{cor}\label{energyEqual}
	For each $\varphi$ on $X$ that corresponds to some $\TT$-invariant Fubini-Study metric $\phi$ on $L$ as above, we have
	\[E^{\na}(\varphi) = E^{\na}(\phi):=\frac {1}{n\vol(L)}\sum_{j=0}^{n-1} \int_{V^{\an}} (\phi - \phi_{\mathrm{triv}}) (d'd''\phi)^j\wedge (d'd''\phi_{\mathrm{triv}})^{n-1-j},\]
	where $V^{\an}$ is identified with the subset $X_\xi$ in $X^{\an, T}$.
\end{cor}
\begin{proof}
	This follows from the previous proposition and Lemma \ref{volEqual}.
\end{proof}

\subsection{Monge-Amp\`ere energy on $X^{\an}$}
We are now ready to prove Theorem \ref{main}.
\begin{thm}
	For a Fubini-Study function $\varphi$ induced by a finitely generated filtration $\cF$, we have
	\[E^{\na}(\varphi) = S(\cF),\tag{1}\]
	and when the filtration comes from a test configuration $(\cX,\xi, \eta)$ with $\eta$ in the Reeb cone, 
	\[E^{\na}(\varphi) =\frac{1}{n} \frac{D_{-\eta} \vol_{\cX_0}(\xi)}{\vol(\xi)}\tag{2}.\]
\end{thm}

\begin{proof} This is already shown in Corollary \ref{energyEqual} and Lemma \ref{volEqual} when $\xi$ is rational.  Indeed, we may assume $\xi$ is integral by homogeneity, that $\tilde{S}(\cF) = E^{\na}(\phi)$ was proved in \cite{BJNA} for a finitely generated filtration, and we have
\[E^{\na}(\varphi)  = E^{\na}(\phi) = \tilde{S}(\cF) = S(\cF).\]
When the filtration comes from a test configuration, Theorem \ref{MAenergy} implies (2).
When $\xi$ is irrational, we can approximate $\xi$ by rational $\xi_k$'s. The conclusion follows by noting $E^{\na}(\varphi_k, \varphi_{\xi_k}) \to E^{\na}(\varphi, \varphi_{\xi})$ %(lemme 8.3.1) 
and similarly for the right hand side.
\end{proof}

\section{The toric case}
In this section, we consider the case when $X$ is toric, and $\TT$ is the torus densely embedded in $X$. We will follow terminology and notations used in \cite{FultonToric}.
Let $X = \Spec R$ is a toric variety corresponding to a cone $\sigma$ in a lattice $N$. Denote by $M$ the dual lattice, and $\sigma^{\vee}$ the dual cone. Then the weight decomposition on $R$ induced by $\TT$ is
$$R=\bigoplus_{u\in \sigma^\vee \cap M}\CC\cdot \chi^u.$$
A Reeb field $\xi$ can be thought of as a vector in $\mathrm{int}( \sigma) $.

To see what Okounkov body looks like on $X$, we first choose a rank $n$ valuation 
$$\nu: R\to \ZZ^n_+, \sum_u a_u \chi^u\mapsto \min_u\{(\la u, e_1\ra, \cdots, \la u, e_n\ra)\},$$
where $e_1, \cdots, e_n\in \sigma$ form a basis of $N_\RR:= N\otimes \RR$ with $\det\{e_i\} = 1$. This also defines a linear map 
$$\iota: M_\RR\to \RR^n, u\mapsto \nu(u) = (\la u, e_1\ra, \cdots, \la u, e_n\ra).$$ 
Let 
$$Q_\xi = \{u\in \sigma^\vee : \la u, \xi\ra \leq 1\},$$ and 
$$P_\xi = \{u\in \sigma^\vee : \la u, \xi\ra = 1\}.$$
\begin{prop}
	The Okounkov body $\Delta$ for $X$ with respect to the valuation $\nu$ is $\iota(Q_\xi)$.
\end{prop}

\begin{proof}
By construction, $\Gamma_m = \nu(R_m) = \{\iota(u) : \la u, \xi\ra \leq m, u\in \sigma^\vee\cap M\} = \iota(mQ_\xi \cap M).$ Let $\mathrm{Conv}(\Gamma_m)$ be the convex hull of $\Gamma_m$. Then for each $m$, 
\[\frac 1m \mathrm{Conv}(\Gamma_m) \subseteq \iota(Q_\xi\cap \frac 1m M).\]
Hence
\[\Delta = \overline{\bigcup_m \frac 1m \mathrm{Conv}(\Gamma_m)} \subseteq \iota(Q_\xi).\]

Conversely, note that
$$\iota\left(\bigcup_m (Q_\xi\cap \frac 1m M)\right)\subseteq \Delta,$$
we have $\Delta = \iota(Q_\xi)$ since $\iota\left(\bigcup_m (Q_\xi\cap \frac 1m M)\right)$ is dense in $\iota(Q_\xi)$.

\end{proof}

A filtration $\cF$ on $R$ induces a function $\psi: \sigma^\vee \cap M \to \RR$, $\psi(u) = \sup\{\lambda: \cF^\lambda R_u \neq 0\}$. By the multiplicativity of the filtration, $\psi$ is superadditive, i.e. $\psi(u_1+u_2)\ge \psi(u_1)+\psi(u_2)$. Define $\tilde{\psi}(u) = \lim\limits_{m\to \infty}\frac 1m \psi(mu)$ to be its homogenization. Then $\tilde{\psi}$ extends to a sublinear function on $\sigma^\vee$, which we will still denote by $\tilde{\psi}$. This function is also concave on $\mathrm{int}(\sigma^\vee)$.

\begin{prop}
	The concave transform $G: \Delta \to \RR_+$ is given by $x=\iota(u)\mapsto \tilde{\psi}(u)$ under the identification in the previous proposition.
\end{prop}

\begin{proof}
For each $0 < t < T(\cF)$, 
$$R_m^t = \bigoplus_{u\in mQ_\xi \cap M, \psi(u)>mt} \CC \cdot \chi^u.$$ Thus for $m \gg 0$, one has 
$\mathrm{Conv}(\nu(R_m^t))\subseteq \iota(mQ_\xi\cap \{\psi>mt\})$. This implies $G\leq \tilde{\psi}$.

For the other direction, it's enough to show that if $u\in Q_\xi\cap M_\QQ$, then $G(\iota(u))\ge\tilde{\psi}(u)$. Indeed, let $u\in Q_\xi\cap M_\QQ$, and suppose $\tilde{\psi}(u) >t$. Then for any $s<t$, there is some $m$ divisible enough such that $mu \in mQ_\xi\cap M$ and $\frac{\psi(mu)}{m} >s$. Thus $mu\in R_m^s$, and $\iota(u)\in \Delta^s$. Therefore $G(\iota(u)) \ge s$ for any $s<t$, and hence $G(\iota(u))\ge\tilde{\psi}(u)$.

\end{proof}

Now we are ready to show Theorem \ref{1.2}.
\begin{thm}
	The Monge-Amp\`ere energy for the filtration $\cF$ is given by 
	$$E^{\mathrm{NA}}(\psi) = \frac{1}{(n+1)\vol(\xi)}  \int_{P_\xi} \tilde{\psi} d\lambda,$$
	where $d\lambda$ is a choice of Lebesgue measure on $P_\xi$ with $\vol(P_\xi) = \vol(\xi)$.
\end{thm}
\begin{proof}
	It follows from discussions in the previous section that 
	\begin{align*}
		E^{\mathrm{NA}}(\psi) &= \frac 1 {\vol(\xi)} \int_\Delta G d\rho = \frac{\det\{(e_i)_{i=1}^n\}}{\vol(\xi)} \int_{Q_\xi}\tilde{\psi}(u) d\rho(u) \\
		&=\frac{\det\{(e_i)_{i=1}^n\}}{\vol(\xi)}\int_0^1\int_{P_\xi}\tilde{\psi}(sw)s^{n-1}d\lambda(w)ds\\
		&= \frac{1}{(n+1)\vol(\xi)} \int_{P_\xi}\tilde{\psi} d\lambda.
	\end{align*}
\end{proof}

\begin{remark} In the toric setting, we have a canonical embedding $j: N_\RR \hookrightarrow X^{\an}$ as $\TT(\KK)$-invariant points. The function $ \varphi = \sup\{\frac{\log|f|+\tilde{\psi}(f)}{\la\alpha, \xi\ra}: f_j\in R_\alpha\}$ is a Fubini-Study function on $X^{\an}$. When restricted to $N_\RR$, $\varphi$ takes the form 
$$\varphi|_{N_\RR}(v) = \sup \{\frac{-\la u, v\ra+\tilde{\psi}(u)}{\la u, \xi\ra}: u\in M_\RR\} = \sup_{u\in P_\xi} \{\la u, v\ra+\tilde{\psi}(u)\}$$
for some monomials $\chi^{u_k}$. It's then not hard to see that $\varphi|_{N_\RR}$ is the Legendre transform of $\tilde{\psi}|_{P_\xi}$.
\end{remark}

\begin{remark}
One can show that 
\[(d'd''\varphi)^{n-1}\wedge d'd''\varphi_\xi^+ = j_*\mathrm{MA}_\RR(\varphi)\]
where $\mathrm{MA}_\RR$ denotes the real Monge-Amp\`ere measure, see \cite{realMA, GGJK21} for more details.
\end{remark}

\bibliography{refse}
\bibliographystyle{alpha}

\Addresses
\end{document}